\documentclass[reqno]{amsart}
\usepackage{amsmath}
\usepackage{amssymb}
\usepackage{hyperref}
\usepackage{a4wide}
\hypersetup{ colorlinks = true, urlcolor = blue, linkcolor = blue, citecolor = red }
\usepackage{xcolor}
\usepackage{enumitem}
\usepackage{esint}
\makeatletter
\def\namedlabel#1#2{\begingroup
	#2%
	\def\@currentlabel{#2}%
	\phantomsection\label{#1}\endgroup
}
\usepackage{varwidth}
\usepackage{tasks}
\DeclareMathOperator{\dv}{div}
\DeclareMathOperator{\loc}{loc}

\newcommand{\RR}{\mathbb{R}}
\newcommand{\mA}{\mathcal{A}}
\newcommand{\Om}{\Omega}
\newcommand{\na}{\nabla}
\newcommand{\pa}{\partial}
\newcommand{\La}{\Lambda}

\newcommand{\ep}{\epsilon}

\newcommand{\la}{\lambda}
\newcommand{\sig}{\sigma}
\newcommand{\cv}{\kappa}
\newcommand{\laz}{\lambda_{\mz}}
\newcommand{\mz}{w}
\newcommand{\mw}{v}
\newcommand{\data}{\mathit{data}}

\theoremstyle{plain}
\newtheorem{theorem}{Theorem}[section]
\newtheorem{lemma}[theorem]{Lemma}

\newtheorem{definition}[theorem]{Definition}

\newtheorem{remark}[theorem]{Remark}
\def\Xint#1{\mathchoice
	{\XXint\displaystyle\textstyle{#1}}%
	{\XXint\textstyle\scriptstyle{#1}}%
	{\XXint\scriptstyle\scriptscriptstyle{#1}}%
	{\XXint\scriptstyle\scriptscriptstyle{#1}}%
	\!\int}
\def\XXint#1#2#3{{\setbox0=\hbox{$#1{#2#3}{\int}$}
		\vcenter{\hbox{$#2#3$}}\kern-.5\wd0}}

\def\Yint#1{\mathchoice
	{\YYint\displaystyle\textstyle{#1}}%
	{\YYint\textstyle\scriptstyle{#1}}%
	{\YYint\scriptstyle\scriptscriptstyle{#1}}%
	{\YYint\scriptscriptstyle\scriptscriptstyle{#1}}%
	\!\iint}
\def\YYint#1#2#3{{\setbox0=\hbox{$#1{#2#3}{\iint}$}
		\vcenter{\hbox{$#2#3$}}\kern-.51\wd0}}
\def\longdash{{-}\mkern-3.5mu{-}} 
\def\fiint{\Yint\longdash}

\def\Xint#1{\mathchoice
	{\XXint\displaystyle\textstyle{#1}}%
	{\XXint\textstyle\scriptstyle{#1}}%
	{\XXint\scriptstyle\scriptscriptstyle{#1}}%
	{\XXint\scriptscriptstyle\scriptscriptstyle{#1}}%
	\!\int}
\def\XXint#1#2#3{{\setbox0=\hbox{$#1{#2#3}{\int}$ }
		\vcenter{\hbox{$#2#3$ }}\kern-.6\wd0}}

\def\dashint{\Xint-}

\DeclareMathOperator{\diam}{diam}

\usepackage{nameref}
\makeatletter
\let\orgdescriptionlabel\descriptionlabel
\renewcommand*{\descriptionlabel}[1]{%
	\let\orglabel\label
	\let\label\@gobble
	\phantomsection
	\edef\@currentlabel{#1}%
	\let\label\orglabel
	\orgdescriptionlabel{#1}%
}
\makeatother
\numberwithin{equation}{section}
\def\Xint#1{\mathchoice
    {\XXint\displaystyle\textstyle{#1}}%
    {\XXint\textstyle\scriptstyle{#1}}%
    {\XXint\scriptstyle\scriptscriptstyle{#1}}%
    {\XXint\scriptscriptstyle\scriptscriptstyle{#1}}%
    \!\int}
\def\XXint#1#2#3{\setbox0=\hbox{$#1{#2#3}{\int}$}
    \vcenter{\hbox{$#2#3$}}\kern-0.5\wd0}
\def\fint{\Xint-}
\def\dashint{\Xint{\raise4pt\hbox to7pt{\hrulefill}}}

\def\XXiint#1#2#3{\setbox0=\hbox{$#1{#2#3}{\iint}$}
    \vcenter{\hbox{$#2#3$}}\kern-0.5\wd0}

\begin{document}
	
\title[Higher integrability for degenerate systems]{Gradient higher integrability for degenerate parabolic double-phase systems}

\author{Wontae Kim}
\address[Wontae Kim]{Research Institute of Mathematics, Seoul National University, Seoul 08826, Korea.}
\email{m20258@snu.ac.kr}
%\thanks{???}

\author{Juha Kinnunen}
\address[Juha Kinnunen]{Department of Mathematics, Aalto University, P.O. BOX 11100, 00076 Aalto, Finland}
\email[Corresponding author]{juha.k.kinnunen@aalto.fi}
%\thanks{??}

\author{Kristian Moring}
\address[Kristian Moring]{Department of Mathematics, Aalto University, P.O. BOX 11100, 00076 Aalto, Finland}
\email{kristian.moring@aalto.fi}
%\thanks{??}

\everymath{\displaystyle}

\makeatletter
\@namedef{subjclassname@2020}{\textup{2020} Mathematics Subject Classification}
\makeatother

\begin{abstract}
We prove a local higher integrability result for the gradient of a weak solution to degenerate parabolic double-phase systems of $p$-Laplace type. 
This result comes with reverse H\"older type estimates. 
The proof is based on a careful phase analysis, estimates in the intrinsic geometries and stopping time arguments.
\end{abstract}

\maketitle
%\tableofcontents
\section{Introduction}
This paper discusses the local higher integrability of the spatial gradient of a weak solution to a double-phase parabolic system
\begin{align}\label{sec1:1}
	u_t-\dv\mA(z,\na u)=-\dv(|F|^{p-2}F+a(z)|F|^{q-2}F)
\text{ in $\Om_T$},
\end{align}
where $z=(x,t)$, $\Omega_T=\Omega\times(0,T)$ is a space-time cylinder with a bounded open set $\Omega\subset\RR^n$ for $n\ge2$ and $2\le p<q<\infty$. 
Here $\mA(z,\na u):\Omega_T\times \RR^{Nn}\longrightarrow \RR^{Nn}$ with $N\ge1$ is a Carath\'eodory vector field satisfying the following structure assumptions: There exist constants $0<\nu\le L<\infty$ such that
	\begin{equation}\label{sec1:2}
			\mA(z,\xi)\cdot \xi\ge \nu(|\xi|^p+a(z)|\xi|^q)\quad\text{and}\quad
			|\mA(z,\xi)|\le L(|\xi|^{p-1}+a(z)|\xi|^{q-1})
	\end{equation}
for almost every $z\in \Omega_T$ and every $\xi\in \RR^{Nn}$. Furthermore, the source term $F:\Omega_T\longrightarrow\RR^{Nn}$ satisfies
\begin{align*}
	\iint_{\Omega_T} H(z,|F|)\ dz=\iint_{\Om_T}(|F|^p+a(z)|F|^q)\,dz<\infty.
\end{align*}
Here we denote $H(z,s):\Omega_T\times \RR^{+ }\longrightarrow \RR^+$,
\begin{align*}
	H(z,s)=s^p+a(z)s^q.
\end{align*}
We assume that the non-negative coefficient function $a:\Omega_T\longrightarrow \RR^+$ satisfies
\begin{align}\label{sec1:4}
	q\le p+\frac{2\alpha}{n+2}
	\quad\text{and}\quad 
	a\in C^{\alpha,\frac\alpha2}(\Omega_T)\text{ for some }\alpha\in(0,1].
\end{align}
Here $a\in C^{\alpha,\frac\alpha2}(\Omega_T)$ means that $a\in L^{\infty}(\Omega_T)$ and that there exists a constant $[a]_{\alpha,\frac\alpha2;\Omega_T}<\infty$ such that
\begin{align*}
	\begin{split}
		|a(x,t)-a(y,t)|\le [a]_{\alpha,\frac\alpha2;\Omega_T}|x-y|^\alpha\quad\text{and}\quad
		|a(x,t)-a(x,s)|\le [a]_{\alpha,\frac\alpha2;\Omega_T}|t-s|^\frac{\alpha}{2}
	\end{split}
\end{align*}
for every $(x,y)\in\Omega$ and $(t,s)\in (0,T)$. 
For short we denote $[a]_{\alpha}=[a]_{\alpha,\frac\alpha2;\Omega_T}$.

We summarize the existing related results in the elliptic and parabolic cases. 
The elliptic double-phase system 
\begin{align*}
		-\dv(|\na u|^{p-2}\na u+a(x)|\na u|^{q-2}\na u)
		=-\dv(|F|^{p-2}F+a(x)|F|^{q-2}F)
\end{align*}
in $\Om$, where 
\begin{align}\label{sec1:7}
	1<p<q\le p+\frac{\alpha p}{n}
	\quad\text{and}\quad 
	a(x)\in C^{\alpha}(\Om)\text{ for some }\alpha\in(0,1],
\end{align}
models a class $(p,q)$-growth problems related to strongly anisotropic materials in the contexts of homogenization and nonlinear elasticity, see \cite{MR864171,MR1209262,MR1329546}.
The proper function space for weak solutions is 
$u\in W^{1,1}(\Om,\RR^N)$ with 
\begin{align*}
		\int_{\Omega} H(x,|\nabla u|)\, dx=\int_{\Omega}(|\nabla u|^p+a(x)|\nabla u|^q)\, dx<\infty.
	\end{align*}
Under \eqref{sec1:7} it has been proved that $|\na u|\in L_{\loc}^q(\Om)$ in \cite{MR2076158} (see also \cite{MR969900,MR1094446} for the $(p,q)$-growth problems).  Harnack's inequality, H\"older continuity, gradient H\"older continuity, gradient higher integrability and Calder\'on-Zygmund type estimates have been discussed in \cite{MR3348922,MR3294408,MR3447716,MR3985927} (see also \cite{MR4397041,MR4467321}). 
For applications and more information, we refer to \cite{MR4258810,MR1810360}. 
A standard approach in the elliptic double-phase systems is to consider two cases: For each ball $B_r(x_0)\subset\Om$, either 
\begin{align}\label{sec1:8}
	\inf_{B_{r}(x_0)}a(x)\le [a]_{\alpha }r^\alpha\quad\text{or}\quad \inf_{B_{r}(x_0)}a(x)> [a]_{\alpha}r^\alpha.
\end{align}
The first condition in \eqref{sec1:8} is called the $p$-phase and in this case the behavior is similar to the $p$-Laplace systems in $B_r(x_0)$. 
The second condition in \eqref{sec1:8} implies that
\begin{align}\label{sec1:9}
	\sup_{B_{r}(x_0)}a(x)< 2\inf_{B_r(x_0)}a(x)
\end{align}
and this leads to the behavior similar to the $(p,q)$-Laplace systems in $B_r(x_0)$.
For this reason the second condition in \eqref{sec1:8} is called the $(p,q)$-phase. 

Parabolic double-phase problems have not been investigated until very recently. 
The existence of weak solutions to \eqref{sec1:1} has been considered in \cite{MR3985549,MR3532237}.
These results seem to cover different ranges of exponents already in the stationary case, see \cite{MR4487513}.
It has been proved in \cite{MR3532237} by using the difference quotient method that $|\na u|\in L^q_{\loc}(\Om_T)$ under appropriate structural assumptions (see also \cite{MR4477803,MR3102165,MR3073153,MR4150873,MR4065088} for the $(p,q)$-growth problems).

The main result of this paper is the following a priori estimate for the gradient of a weak solution to \eqref{sec1:1}.
We denote $Q_{r}(z_0)=B_{r}(x_0)\times (t_0-r^2,t_0+r^2)$ and 
  \begin{align*}
		\mathit{data}=(n,N,p,q,\alpha,\nu,L,[a]_{\alpha},\diam(\Omega),\|u\|_{L^\infty(0,T;L^2(\Omega))},\|H(z,|\na u|)\|_{L^1(\Omega_T)},\|H(z,|F|)\|_{L^1(\Omega_T)}).
\end{align*}

\begin{theorem}\label{main_theorem}
	Assume that \eqref{sec1:2} and \eqref{sec1:4} hold true and let $u$ be a weak solution to \eqref{sec1:1}. 
	Then there exist constants $0<\ep_0=\ep_0(\mathit{data})$ and $c=c(\mathit{data},\lVert a\rVert_{L^\infty(\Om_T)})\ge1$ such that
	\begin{align*}
		\begin{split}
			&\fiint_{Q_{r}(z_0)}H(z,|\na u|)^{1+\ep}\,dz
			\le c \left(\fiint_{Q_{2r}(z_0)}H(z,|\na u|)\,dz\right)^{1+\frac{\ep q}{2}}\\
			&\qquad\qquad\qquad+c\left(\fiint_{Q_{2r}(z_0)}(H(z,|F|)+1)^{1+\ep}\,dz\right)^{\frac{q}{2}}
		\end{split}
	\end{align*}
	 for every $Q_{2r}(z_0)\subset\Om_T$ and $\ep\in(0,\ep_0)$.
\end{theorem}

As far as we are aware this is the first regularity result for parabolic double-phase problems under the general structural conditions \eqref{sec1:2} and \eqref{sec1:4}. We consider weak solutions that satisfy a technical assumption $|\na u|\in L^q(\Om_T)$, see Definition \ref{weak_solution}.
It is also possible to obtain the main result under the assumption
\begin{align*}
		\iint_{\Omega_T} H(z,|\na u|)\, dz=\iint_{\Omega_T}(|\nabla u|^p+a(z)|\nabla u|^q)\,dz<\infty,
	\end{align*}
by applying a parabolic Lipschitz truncation, see Remark \ref{weak_solution_remark}.
This technique is out of the scope of this paper and it is discussed in \cite{KKS}.
This extends the corresponding results for the $p$-Laplace systems $(a(z)\equiv0)$ in \cite{MR1749438}, 
where a reverse H\"older inequality for the gradient has been proved in $p$-intrinsic cylinders
\begin{align*}
Q_{\rho}^\la(z_0)=B_{\rho}(x_0)\times (t_0-\la^{2-p}\rho^2,t_0+\la^{2-p}\rho^2)
\end{align*}
by using parabolic Caccioppoli and Poincar\'e inequalities and a stopping time argument. 
Appropriate intrinsic cylinders have to be considered for other parabolic systems. The parabolic $(p,q)$-Laplace system ($a(z)\equiv a_0$ for some constant $a_0>0$) was considered in \cite{MR4302665} and the gradient reverse H\"older inequality was proved in $(p,q)$-intrinsic cylinders
\begin{align*}
G_{\rho}^\la(z_0)=B_{\rho}(x_0)\times \left(t_0-\tfrac{\la^2}{\la^p+a_0\la^q}\rho^2,t_0+\tfrac{\la^2}{\la^p+a_0\la^q}\rho^2\right).
\end{align*}
In \cite{MR2779582}, the gradient higher integrability result has been discussed for the parabolic $p(\cdot)$-Laplace type system
\begin{align*}
	u_t-\dv( |\na u|^{p(z)-2}\na u)=-\dv(|F|^{p(z)-2}F)
\end{align*}
in $\Om_T$, where $p(\cdot):\Om_T \longrightarrow  \RR^+$ is a continuous function with
\begin{align*}
	\frac{2n}{n+2}<\inf_{z\in \Om_T}p(z)\le p(\cdot)\le \sup_{z\in \Om_T}p(z)<\infty
\end{align*}
and $p(\cdot)$ satisfies a logarithmic modulus of continuity condition. In this case the intrinsic cylinders are of the form
\begin{align*}
	B_{\rho}(x_0)\times \bigl(t_0-\la^{\frac{2-p(z_0)}{p(z_0)}}\rho^2,t_0+\la^\frac{2-p(z_0)}{p(z_0)}\rho^2\bigr).
\end{align*}

Several new features appear in the parabolic double-phase problem ~\eqref{sec1:1} compared to the $p$-Laplace systems in \cite{MR1749438} and to the $(p,q)$-case in \cite{MR4302665}.
The first novelty in our argument is that we provide a new criterion replacing \eqref{sec1:8} in order to be able to adopt the stopping time argument with intrinsic cylinders in \cite{MR1749438}.
For each point 
\[
z_0\in\{ z\in \Omega_T:|\na u(z)|^p+a(z)|\na u(z)|^q>\La\},
\]
we consider $\la=\la(z_0)>0$ such that $\La=\la^p+a(z_0)\la^q$.
Employing the fact that $s \to s^p + a(z_0)s^q$  is strictly increasing and
noting that $z_0\in\{ z\in \Omega_T:|\na u(z)|^p>\la^p\}$,
 we may apply a stopping time argument with the $p$-intrinsic cylinders. 
The second novelty is to consider two alternatives: For $K>1$, either 
\begin{align*}
	K\la^p\ge a(z_0)\la^q\quad\text{or}\quad K\la^p\le a(z_0)\la^q.
\end{align*}
These are called $p$-intrinsic and $(p,q)$-intrinsic cases, respectively. The $p$-intrinsic case is related to the $p$-Laplace systems and the $(p,q)$-intrinsic case is related to the $(p,q)$-problems. For the double-phase problems we have to consider both of them. We are convinced that this technique will be useful in other regularity results for parabolic doubly-nonlinear problems.
In the $p$-intrinsic case, it is possible to obtain the reverse H\"older inequality in the $p$-intrinsic cylinders as in \cite{MR1749438}.
Roughly speaking, we have
\begin{align*}
	|\na u|^{p-2}\na u+a(z)|\na u|^{q-2}\na u&\approx |\na u|^{p-2}\na u+K\la^{p-q}|\na u|^{q-2}\na u
	\approx |\na u|^{p-2}\na u
\end{align*}
in the stopping time argument with a $p$-intrinsic cylinder. On the other hand, the $(p,q)$-intrinsic case implies the second condition in \eqref{sec1:8} for a sufficiently large $K>1$. 
This leads \eqref{sec1:9} and we have
\begin{align*}
	|\na u|^{p-2}\na u+a(z)|\na u|^{q-2}\na u\approx |\na u|^{p-2}\na u+a(z_0)|\na u|^{q-2}\na u
\end{align*}
in the stopping time argument with $p$-intrinsic cylinder.
Consequently, we may apply the $(p,q)$-intrinsic cylinders to obtain the reverse H\"older inequality. Note that 
\[
z_0\in \{ z\in \Omega_T:|\na u(z)|^p+a(z)|\na u(z)|^q>\la^p+a(z_0)\la^q\}
\]
and $G_\rho^\la(z_0)\subset Q_\rho^\la(z_0)$ with $a_0=a(z_0)$. Thus, it is possible to obtain a stopping time argument in the $(p,q)$-intrinsic cylinders from the stopping time argument in the $p$-intrinsic cylinders.  
Finally, the continuity of $a(\cdot)$ implies the continuity of $\la(\cdot)$ and this enables us to prove a Vitali type covering lemma. The desired estimate follows by using Fubini's theorem.

\section{Energy estimates}

We apply the following definition of weak solution.
\begin{definition}\label{weak_solution}
	A function $u:\Om_T\longrightarrow\RR^N$ with
	\begin{align*}
			u\in C(0,T;L^2(\Om,\RR^N))\cap L^q(0,T;W^{1,q}(\Om,\RR^N))
	\end{align*}
is a weak solution to \eqref{sec1:1}, if
\begin{align*}
	\iint_{\Om_T}(-u\cdot\varphi_t+\mA(z,\na u)\cdot \na \varphi)\,dz=\iint_{\Om_T}(|F|^{p-2}F\cdot\na \varphi+a(z)|F|^{p-2}F\cdot \na\varphi)\,dz
\end{align*}
for every $\varphi\in C_0^\infty(\Omega_T,\RR^N)$.
\end{definition}

\begin{remark}\label{weak_solution_remark}
	A more standard assumption on the function space would be
 \begin{align*}
			u\in C(0,T;L^2(\Om,\RR^N))\cap L^1(0,T;W^{1,1}(\Om,\RR^N))
	\end{align*}
 with
	\begin{align*}
		\iint_{\Omega_T} H(z,|\na u|)\, dz=\iint_{\Omega_T}(|\nabla u|^p+a(z)|\nabla u|^q)\,dz<\infty.
	\end{align*}
	However, this assumption does not seem to be enough in the proof of the energy estimate using the Steklov averages, see Lemma \ref{sec3:lem:1} below. This unexpected challenge does not occur in the elliptic case, since the mollification in time is not needed.
    It is possible to derive Lemma~\ref{sec3:lem:1} under the natural function space assumption above by a parabolic Lipschitz truncation technique, see \cite{KKS}.  We emphasize that the assumption $|\na u|\in L^q(\Om_T)$ is only applied in the proof of Lemma~\ref{sec3:lem:1} and it is not needed in the rest of the paper. With this observation Theorem~\ref{main_theorem} holds true also under the natural function space assumption above.
\end{remark}

In the rest of this section, we provide three energy estimates. The first lemma is a parabolic double-phase Caccioppoli inequality.
In general, the time derivative of a weak solution does not belong to $L^2$ and does not even exist a priori. 
To be able to derive a suitable energy estimate, we use the following mollification in time. We define the Steklov average $f_h$, with $0<h<T$, of $f\in L^1(\Omega_T)$ by
\begin{align*}
	f_h(x,t)=
	\begin{cases}
		\fint_t^{t+h}f(x,s)\,ds,&0<t<T-h,\\
		0,&T-h\le t.
	\end{cases}
\end{align*}
For the properties of Steklov averages, we refer to~\cite{MR1230384}.

We apply the following notation. 
A space-time cylinder in $\RR^{n+1}$ is denoted by 
\[
Q_{R,\ell}(z_0)=B_{R}(x_0)\times (t_0-\ell,t_0+\ell),
\quad r>0,\quad l>0,
\]
 and the integral average of $u$ over $Q_{R,\ell}(z_0)$ is denoted by
 \begin{align*}
     u_{Q_{R,\ell}(z_0)}=\fiint_{Q_{R,\ell}(z_0)}u\,dz.
 \end{align*}

\begin{lemma}\label{sec3:lem:1}
	Let $u$ be a weak solution to \eqref{sec1:1}. Then there exists a constant $c=c(n,p,q,\nu,L)$ such that
	\begin{align*}
		\begin{split}
			&\sup_{t\in (t_0-\tau,t_0+\tau)}\fint_{B_{r}(x_0)}\frac{|u-u_{Q_{r,\tau}(z_0)}|^2}{\tau}\,dx+\fiint_{Q_{r,\tau}(z_0)}H(z,|\na u|)\,dz\\
			&\qquad\le c\fiint_{Q_{R,\ell}(z_0)}\left(\frac{|u-u_{Q_{R,\ell}(z_0)}|^p}{(R-r)^p}+a(z)\frac{|u-u_{Q_{R,\ell}(z_0)}|^q}{(R-r)^q}\right)\,dz\\
			&\qquad\qquad+c\fiint_{Q_{R,\ell}(z_0)}\frac{|u-u_{Q_{R,\ell}(z_0)}|^2}{\ell-\tau}\,dz+c\fiint_{Q_{R,\ell}(z_0)}H(z,|F|)\,dz
		\end{split}
	\end{align*}
 for every $Q_{R,\ell}(z_0)\subset\Omega_T$, with $R,\ell>0$, $r\in [R/2,R)$ and $\tau\in[\ell/2^2,\ell)$.
\end{lemma}

\begin{proof}
	Let $\eta\in C_0^\infty(B_{R}(x_0))$ be a cut-off function with
	\begin{align}\label{sec3:1}
		0\le \eta\le 1,\quad\eta\equiv1\text{ in } B_{r}(x_0)\quad\text{and}\quad\lVert \na \eta\rVert_{L^\infty}\le \frac{2}{R-r}.
	\end{align}
	For $\tau\in[\ell/2^2,\ell)$, let $h_0>0$ be sufficiently small so that there exists a cut-off function $\zeta\in C_0^\infty(I_{\ell-h_0}(t_0))$ with
	\begin{align}\label{sec3:3}
			0\le \zeta\le 1,\quad \zeta\equiv1\text{ in } I_{\tau}(t_0)\quad\text{and}\quad\lVert\pa_t\zeta\rVert_{L^\infty}\le \frac{3}{\ell-\tau}.
	\end{align}
Let $t_*\in I_{\tau}(t_0)$ and $\delta\in(0,h_0)$. We define $\zeta_\delta$ as
\begin{align}\label{sec3:4}
	\zeta_{\delta}(t)=
	\begin{cases}
		1,&t\in (-\infty,t_*-\delta),\\
		1-\frac{t-t_*+\delta}{\delta},&t\in[t_*-\delta,t_*],\\
		0,&t\in(t_*,\infty).
	\end{cases}
\end{align}

For $h\in(0,h_0)$, we consider \eqref{sec1:1} in terms of Steklov averages and obtain
\begin{equation}\label{sec3:2}
	\pa_t[u-u_{Q_{R,\ell}(z_0)}]_h-\dv [\mA(\cdot,\na u)]_h\
	=-\dv[|F|^{p-2}F+a|F|^{q-2}F]_h
\end{equation}
in $B_{R}(x_0)\times I_{\ell-h}(t_0)$.
Then we observe that
\begin{align*}
[u-u_{Q_{R,\ell}(z_0)}]_h\eta^q\zeta^2 \zeta_\delta\in
   	 W^{1,2}_0(I_{\ell-h};L^2(B_{R}(x_0),\RR^N))\cap L^q(I_{\ell-h};W_0^{1,q}(B_{R}(x_0),\RR^N)).
\end{align*}
By applying $\varphi=[u-u_{Q_{R,\ell}(z_0)}]_h\eta^q\zeta^2\zeta_{\delta}$ as a test function in \eqref{sec3:2}, we have
\begin{align}\label{sec3:6}
	\begin{split}
		\mathrm{I}+\mathrm{II}&=
		\fiint_{Q_{R,\ell}(z_0)}\pa_t[u-u_{Q_{R,\ell}}(z_0)]_h\cdot\varphi\,dz
		+\fiint_{Q_{R,\ell}(z_0)}[\mA(\cdot,\na u)]_h\cdot\na\varphi \,dz\\
		&=\fiint_{Q_{R,\ell}(z_0)}[|F|^{p-2}F+a|F|^{q-2}F]_h\cdot \na\varphi\,dz=\mathrm{III}.
	\end{split}
\end{align}
We show that $\mathrm{II}$ and $\mathrm{III}$ are finite under the assumption $|\na u|\in L^q(\Om_T)$. 
The structural assumptions and the properties of the Steklov average lead to
	\begin{align*}
			\mathrm{II}
			\le L\fiint_{Q_{R,\ell}(z_0)}|(|\nabla u|^{p-1})_h(x,t)||\nabla\varphi(x,t)|\, dx\ dt+L\fiint_{Q_{R,\ell}(z_0)}|(a|\nabla u|^{q-1})_h(x,t)| |\nabla\varphi(x,t)|\, dx\ dt.
	\end{align*}
     The first term on the right-hand side is finite as in the case of the parabolic $p$-Laplace systems. The second term on the right-hand side can be written as
	\begin{align*}
     \begin{split}
     &\fiint_{Q_{R,\ell}(z_0)}|(a|\nabla u|^{q-1})_h(x,t)| |\nabla\varphi(x,t)|\, dx\ dt\\
    & =\fiint_{Q_{R,\ell}(z_0)}\fint_t^{t+h}[a(x,s)]^\frac{q-1}{q}|\nabla u(x,s)|^{q-1}\  [a(x,s)]^\frac{1}{q}|\nabla\varphi(x,t)| \, ds \, dx\ dt.
      \end{split}
	\end{align*} 
	By H\"older's inequality and the properties of the Steklov average, there exists a constant $c=c(n)$ such that
	\begin{align*}
		\begin{split}
			&\fiint_{Q_{R,\ell}(z_0)}|(a|\nabla u|^{q-1})_h(x,t)| |\nabla\varphi(x,t)|\, dx\ dt\\
			&\le c\left(\fiint_{Q_{R,\ell}(z_0)}a(x,t)|\nabla u(x,t)|^{q}  \ dx\ dt\right)^\frac{q-1}{q}\left(\fiint_{Q_{R,\ell}(z_0)}\fint_t^{t+h}a(x,s)|\nabla \varphi(x,t)|^{q}\,ds\,dx\ dt\right)^\frac{1}{q}\\
            &= c\left(\fiint_{Q_{R,\ell}(z_0)}a(x,t)|\nabla u(x,t)|^{q}\, dx\ dt\right)^\frac{q-1}{q}\left(\fiint_{Q_{R,\ell}(z_0)}a_h(x,t)|\nabla \varphi(x,t)|^{q}\, dx\ dt\right)^\frac{1}{q}.
		\end{split}
	\end{align*} 
This shows that $\mathrm{II}$ is finite if $|\na u|\in L^q(\Om_T)$. A similar argument applies for $\mathrm{III}$.

\textbf{Estimate of $\mathrm{I}$:} Integration by parts gives
\begin{align}\label{sec3:7}
	\begin{split}
		\mathrm{I}
		&=\fiint_{Q_{R,\ell}(z_0)}\frac{1}{2}(\pa_t|[u-u_{Q_{R,\ell}(z_0)}]_h|^2)\eta^q\zeta^2\zeta_{\delta}\,dz\\
		&=-\fiint_{Q_{R,\ell}(z_0)}|[u-u_{Q_{R,\ell}(z_0)}]_h|^2\eta^q\zeta\zeta_{\delta}\pa_t\zeta\,dz\\
		&\qquad-\fiint_{Q_{R,\ell}(z_0)}\frac{1}{2}|[u-u_{Q_{R,\ell}(z_0)}]_h|^2\eta^q\zeta^2\pa_t\zeta_{\delta}\,dz.
	\end{split}
\end{align}
We estimate the first term on the right-hand side of \eqref{sec3:7} by \eqref{sec3:3} and obtain
\begin{align*}
		-\fiint_{Q_{R,\ell}(z_0)}|[u-u_{Q_{R,\ell}(z_0)}]_h|^2\eta^q\zeta\zeta_{\delta}\pa_t\zeta\,dz
		\ge-\fiint_{Q_{R,\ell}(z_0)}\frac{|[u-u_{Q_{R,\ell}(z_0)}]_h|^2}{\ell-\tau}\,dz.
\end{align*}
For the second term  on the right-hand side of \eqref{sec3:7}, by \eqref{sec3:4} we have
\begin{align*}
	\begin{split}
		&-\fiint_{Q_{R,\ell}(z_0)}\frac{1}{2}|[u-u_{Q_{R,\ell}(z_0)}]_h|^2\eta^q\zeta^2\pa_t\zeta_{\delta}\,dz\\
		&\qquad=\frac{1}{|Q_{R,\ell}|}\fint_{t_*-\delta}^{t_*}\int_{B_{R}(x_0)}\frac{1}{2}|[u-u_{Q_{R,\ell}(z_0)}]_h|^2\eta^q\zeta^2\,dx\,dt\\
		&\qquad\ge \frac{1}{|Q_{R,\ell}|}\fint_{t_*-\delta}^{t_*}\int_{B_{r}(x_0)}\frac{1}{2}|[u-u_{Q_{R,\ell}(z_0)}]_h|^2\,dx\,dt.
	\end{split}
\end{align*}
Thus, we get
\begin{align*}
	\begin{split}
		\lim_{h\to0^+}\lim_{\delta\to0^+}\mathrm{I}&\ge -\fiint_{Q_{R,\ell}(z_0)}\frac{|u-u_{Q_{R,\ell}(z_0)}|^2}{\ell-\tau}\,dz\\
		&\qquad+\frac{1}{2|Q_{R,\ell}|}\int_{B_r(x_0)}|u(x,t_*)-u_{Q_{R,\ell}(z_0)}|^2\,dx.
	\end{split}
\end{align*}

\textbf{Estimate of $\mathrm{II}$:} It holds that
\begin{align}\label{sec3:11}
	\begin{split}
		\mathrm{II}
		&=\fiint_{Q_{R,\ell}(z_0)}[\mA(\cdot,\na u)]_h\cdot[ \na u]_h\eta^q\zeta^2\zeta_\delta\,dz\\
		&\qquad+q\fiint_{Q_{R,\ell}(z_0)}[\mA(\cdot,\na u)]_h\cdot [u-u_{Q_{R,\ell}(z_0)}]_h\na\eta\eta^{q-1}\zeta^2\zeta_\delta\,dz.
	\end{split}
\end{align}
To estimate the first term in \eqref{sec3:11}, we apply \eqref{sec1:2} to get
\begin{align*}
	\begin{split}
			&\lim_{h\to0^+}\lim_{\delta\to0^+}\fiint_{Q_{R,\ell}(z_0)}[\mA(\cdot,\na u)]_h\cdot [\na u]_h\eta^q\zeta^2\zeta_\delta\,dz	\\
			&\qquad\ge \frac{\nu}{|Q_{R,\ell}|}\int_{I_{\ell}(t_0)\cap (-\infty,t_*)}\int_{B_{R}(x_0)}(|\na u|^p+a(z)|\na u|^q)\eta^q\zeta^2\,dx\,dt.
	\end{split}
\end{align*}
To estimate the second term in \eqref{sec3:11}, we use \eqref{sec1:2} and \eqref{sec3:1} to conclude that
\begin{align*}
	\begin{split}
		&\lim_{h\to0^+}\lim_{\delta\to0^+}q\fiint_{Q_{R,\ell}(z_0)}[\mA(\cdot,\na u)]_h\cdot [u-u_{Q_{R,\ell}(z_0)}]_h\na\eta\eta^{q-1}\zeta^2\zeta_\delta\,dz\\
		&\qquad\ge - \frac{Lq}{|Q_{R,\ell}|}\int_{I_{\ell}(t_0)\cap (-\infty,t_*)}\int_{B_{R}(x_0)}|\na u|^{p-1}\eta^{q-1}\zeta^2\frac{|u-u_{Q_{R,\ell}(z_0)}|}{R-r}\,dx\,dt\\
		&\qquad\qquad-\frac{Lq}{|Q_{R,\ell}|}\int_{I_{\ell}(t_0)\cap (-\infty,t_*)}\int_{B_{R}(x_0)}a(z)|\na u|^{q-1}\eta^{q-1}\zeta^2\frac{|u-u_{Q_{R,\ell}(z_0)}|}{R-r}\,dx\,dt.
	\end{split}
\end{align*}
By Young's inequality, there exists a constant $c=c(p,q,\nu,L)$ such that
\begin{align*}
	\begin{split}
		& \lim_{h\to0^+}\lim_{\delta\to0^+}q\fiint_{Q_{R,\ell}(z_0)}[\mA(z,\na u)]_h\cdot [u-u_{Q_{R,\ell}(z_0)}]_h\na\eta\eta^{q-1}\zeta^2\zeta_\delta\,dz\\
		&\qquad\ge-\frac{\nu}{4|Q_{R,\ell}|}\int_{I_{\ell}(t_0)\cap (-\infty,t_*)}\int_{B_{R}(x_0)}(|\na u|^p+a(z)|\na u|^q)\eta^q\zeta^2\,dx\,dt\\
		&\qquad\qquad-c\fiint_{Q_{R,\ell}(z_0)}\left(\frac{|u-u_{Q_{R,\ell}(z_0)}|^p}{(R-r)^p}+a(z)\frac{|u-u_{Q_{R,\ell}(z_0)}|^q}{(R-r)^q}\right)\,dz.
	\end{split}
\end{align*}
It follows that
\begin{align*}
	\begin{split}
		\lim_{h\to0^+}\lim_{\delta\to0^+}\mathrm{II}&\ge \frac{3\nu}{4|Q_{R,\ell}|}\int_{I_{\ell}(t_0)\cap (-\infty,t_*)}\int_{B_{R}(x_0)}(|\na u|^p+a(z)|\na u|^q)\eta^q\zeta^2\,dx\,dt\\
		&\qquad-c\fiint_{Q_{R,\ell}(z_0)}\left(\frac{|u-u_{Q_{R,\ell}(z_0)}|^p}{(R-r)^p}+a(z)\frac{|u-u_{Q_{R,\ell}(z_0)}|^q}{(R-r)^q}\right)\,dz.
	\end{split}
\end{align*}

\textbf{Estimate of $\mathrm{III}$:} We apply Young's inequality as above and obtain
\begin{align*}
	\begin{split}
		&\lim_{h\to0^+}\lim_{\delta\to0^+}\mathrm{III}
		\le c\fiint_{Q_{R,\ell}(z_0)}|F|^p+a(z)|F|^q\,dz\\
		&\qquad+\frac{\nu}{2|Q_{R,\ell}|}\int_{I_{\ell}(t_0)\cap (-\infty,t_*)}\int_{B_{R}(x_0)}(|\na u|^p+a(z)|\na u|^q)\eta^q\zeta^2\,dx\,dt\\
		&\qquad+c\fiint_{Q_{R,\ell}(z_0)}\left(\frac{|u-u_{Q_{R,\ell}(z_0)}|^p}{(R-r)^p}+a(z)\frac{|u-u_{Q_{R,\ell}(z_0)}|^q}{(R-r)^q}\right)\,dz.
	\end{split}
\end{align*}
By applying the estimates above in \eqref{sec3:6}, we obtain
\begin{align*}
	\begin{split}
		&\frac{1}{|Q_{R,\ell}|}\int_{B_r(x_0)}|u(x,t_*)-u_{Q_{R,\ell}(z_0)}|^2\,dx\\
		&\qquad\qquad+\frac{1}{|Q_{R,\ell}|}\int_{I_{\ell}(t_0)\cap (-\infty,t_*)}\int_{B_{R}(x_0)}(|\na u|^p+a(z)|\na u|^q)\eta^q\zeta^2\,dx\,dt\\
		&\qquad\le c\fiint_{Q_{R,\ell}(z_0)}\left(\frac{|u-u_{Q_{R,\ell}(z_0)}|^p}{(R-r)^p}+a(z)\frac{|u-u_{Q_{R,\ell}(z_0)}|^q}{(R-r)^q}\right)\,dz\\		
		&\qquad\qquad+c\fiint_{Q_{R,\ell}(z_0)}\frac{|u-u_{Q_{R,\ell}(z_0)}|^2}{\ell-\tau}\,dz+c\fiint_{Q_{R,\ell}(z_0)}(|F|^p+a(z)|F|^q)\,dz.
	\end{split}
\end{align*}
Since $t_*\in I_{\ell}(t_0)$ is arbitrary, $|B_{R}|\approx c(n)|B_{r}|$ and $|I_{\ell}|\approx |I_{\tau}|$, we get
\begin{align*}
	\begin{split}
		&\sup_{t\in (t_0-\tau,t_0+\tau)}\fint_{B_{r}(x_0)}\frac{|u-u_{Q_{R,\ell}(z_0)}|^2}{\tau}\,dx+\fiint_{Q_{r,\tau}(z_0)}(|\na u|^p+a(z)|\na u|^q)\,dz\\
		&\qquad\le c\fiint_{Q_{R,\ell}(z_0)}\left(\frac{|u-u_{Q_{R,\ell}(z_0)}|^p}{(R-r)^p}+a(z)\frac{|u-u_{Q_{R,\ell}(z_0)}|^q}{(R-r)^q}\right)\,dz\\
		&\qquad\qquad+c\fiint_{Q_{R,\ell}(z_0)}\frac{|u-u_{Q_{R,\ell}(z_0)}|^2}{\ell-\tau}\,dz+c\fiint_{Q_{R,\ell}(z_0)}(|F|^p+a(z)|F|^q)\,dz.
	\end{split}
\end{align*}
\end{proof}

The second lemma is a gluing lemma, which enables us to estimate integral averages over time-slices.
The spatial integral average of $u$ over $B_R(x_0)$ is denoted by
 \begin{align*}
    u_{B_{R}(x_0)}=u_{B_{R}(x_0)}(t)=\fint_{B_{R}(x_0)}u(x,t)\,dx.
 \end{align*}

\begin{lemma}\label{sec3:lem:2}
	Let $u$ be a weak solution to \eqref{sec1:1} and let $\eta\in C_0^\infty(B_{R}(x_0))$ be a function such that
	\begin{align}\label{sec3:34}
		\eta\ge0,\quad\fint_{B_{R}(x_0)}\eta\,dx=1\quad\text{and}\quad \lVert \eta\rVert_{L^\infty}+R\lVert\na \eta\rVert_{L^\infty}\le c,
	\end{align}
 where $c=c(n)$.
	Then there exists a constant $c=c(n,L)$ such that
	\begin{align*}
		\begin{split}
			&\sup_{t_1,t_2\in (t_0-\ell,t_0+\ell)}|(u\eta)_{B_{R}(x_0)}(t_2)-(u\eta)_{B_{R}(x_0)}(t_1)|\\
			&\qquad\le c\frac{\ell}{R}\fiint_{Q_{R,\ell}(z_0)}(|\na u|^{p-1}+a(z)|\na u|^{q-1})\,dz
			+c\frac{\ell}{R}\fiint_{Q_{R,\ell}(z_0)}(|F|^{p-1}+a(z)|F|^{q-1})\,dz
		\end{split}
	\end{align*}
  for every $Q_{R,\ell}(z_0)=B_{R}(x_0)\times (t_0-\ell,t_0+\ell)\subset\Omega_T$ with $R,\ell>0$.
 \end{lemma}

\begin{proof}
	Let $t_1,t_2\in (t_0-\ell,t_0+\ell)$ with $t_1<t_2$. For $\delta\in(0,1)$ small enough, we define $\zeta_\delta\in W_0^{1,\infty}(t_0-\ell,t_0+\ell)$ by
	\begin{align*}
		\zeta_\delta(t)=
		\begin{cases}
			\qquad 0,&t_0-\ell\le t\le t_1-\delta,\\
			\frac{t-t_1+\delta}{\delta},&t_1-\delta<t<t_1,\\
			\qquad1,&t_1\le t\le t_2,\\
			\frac{t_2+\delta-t}{\delta},&t_2<t<t_2+\delta,\\
			\qquad0,&t_2+\delta\le t\le t_0+\ell.
		\end{cases}
	\end{align*}
	By applying $\eta\zeta_\delta\in W_0^{1,\infty}(Q_{R,\ell}(z_0))$ as a test function in \eqref{sec1:1}, we obtain
	\begin{align*}
		\begin{split}
			&\fint_{t_1-\delta}^{t_1}\fint_{B_{R}(x_0)} u\eta\,dx\,dt-\fint_{t_2+\delta}^{t_2}\fint_{B_{R}(x_0)} u\eta\,dx\,dt\\
			&\qquad\le L\int_{t_1-\delta}^{t_2+\delta}\fint_{B_{R}(x_0)} (|\na u|^{p-1}+a(z)|\na u|^{q-1})|\na \eta||\zeta_\delta|\,dx\,dt\\
			&\qquad\qquad+\int_{t_1-\delta}^{t_2+\delta}\fint_{B_{R}(x_0)}(|F|^{p-1}+a(z)|F|^{q-1})|\na \eta||\zeta_\delta|\,dz.
		\end{split}
	\end{align*}
	Letting $\delta\longrightarrow0^+$ and using the third condition in \eqref{sec3:34}, we obtain
	\begin{align*}
		\begin{split}
			&|(u\eta)_{B_{R}(x_0)}(t_1)-(u\eta)_{B_{R}(x_0)}(t_2)|
			\le c\frac{\ell}{R}\fiint_{Q_{R,\ell}(z_0)}(|\na u|^{p-1}+a(z)|\na u|^{q-1})\,dz\\
			&\qquad\qquad\qquad+c\frac{\ell}{R}\fiint_{Q_{R,\ell}(z_0)}(|F|^{p-1}+a(z)|F|^{q-1})\,dz,
		\end{split}
	\end{align*}
 where $c=c(n,L)$.
	This completes the proof.
\end{proof}

Then we consider a parabolic Poincar\'e inequality.
\begin{lemma}\label{sec3:lem:3}
	Let $u$ be a weak solution to \eqref{sec1:1}. Then  there exists a constant $c=c(n,N,m,L)$ such that
	\begin{align*}
		\begin{split}
			&\fiint_{Q_{R,\ell}(z_0)}\frac{|u-u_{Q_{R,\ell}(z_0)}|^{\theta m}}{R^{\theta m}}\,dz\le c\fiint_{Q_{R,\ell}(z_0)}|\na u|^{\theta m}\,dz\\
			&\qquad+c\left(\frac{\ell}{R^2}\fiint_{Q_{R,\ell}(z_0)}(|\na u|^{p-1}+a(z)|\na u|^{q-1}+|F|^{p-1}+a(z)|F|^{q-1})\,dz\right)^{\theta m}
		\end{split}
	\end{align*}
 for every $Q_{R,\ell}(z_0)=B_{R}(x_0)\times (t_0-\ell,t_0+\ell)\subset\Omega_T$ with $R,\ell>0$, $m\in(1,q]$ and $\theta\in(1/m,1]$,
\end{lemma}

\begin{proof}
	The triangle inequality gives
	\begin{align*}
		\begin{split}
			&\fiint_{Q_{R,\ell}(z_0)}\frac{|u-u_{Q_{R,\ell}(z_0)}|^{\theta m}}{R^{\theta m}}\,dz
			\le c\fiint_{Q_{R,\ell}(z_0)}\frac{|u-u_{B_{R}(x_0)}(t)|^{\theta m}}{R^{\theta m}}\,dz\\
			&\qquad+c\fint_{I_{\ell}(t_0)}\frac{|u_{Q_{R,\ell}(z_0)}-u_{B_{R}(x_0)}(t)|^{\theta m}}{R^{\theta m}}\,dt,
		\end{split}
	\end{align*}
	where $c=c(m)$. By applying the Poincar\'e inequality in the spatial direction, we have
	\begin{align*}
		\begin{split}
			&\fiint_{Q_{R,\ell}(z_0)}\frac{|u-u_{Q_{R,\ell}(z_0)}|^{\theta m}}{R^{\theta m}}\,dz \le c\fiint_{Q_{R,\ell}(z_0)}|\na u|^{\theta m}\,dz\\
			&\qquad+c\fint_{I_{\ell}(t_0)}\frac{|u_{Q_{R,\ell}(z_0)}-u_{B_{R}(x_0)}(t)|^{\theta m}}{R^{\theta m}}\,dt,
		\end{split}
	\end{align*}
where $c=c(n,N,m)$.

	To complete the proof, we estimate the second term on the right-hand side in the estimate above. By H\"older's inequality, we have
	\begin{align*}
			\fint_{I_{\ell}(t_0)}|u_{Q_{R,\ell}(z_0)}-u_{B_{R}(x_0)}(t)|^{\theta m}\,dt
			\le \fint_{I_{\ell}(t_0)}\fint_{I_{\ell}(t_0)}|u_{B_{R}(x_0)}(s)-u_{B_{R}(x_0)}(t)|^{\theta m}\,dt \,ds.
	\end{align*}
	For $\eta\in C_0^\infty(B_R(x_0))$ satisfying \eqref{sec3:34}, it holds that
	\begin{align*}
		\begin{split}
			&\fint_{I_{\ell}(t_0)}\fint_{I_{\ell}(t_0)}|u_{B_{R}(x_0)}(s)-u_{B_{R}(x_0)}(t)|^{\theta m}\,dt \,ds\\
			&\qquad\le c\fint_{I_{\ell}(t_0)}|(u\eta)_{B_{R}(x_0)}(t)-u_{B_{R}(x_0)}(t)|^{\theta m}\,dt\\
			&\qquad\qquad+c\sup_{t,s\in I_{\ell}(t_0)}|(u\eta)_{B_{R}(x_0)}(t)-(u\eta)_{B_{R}(x_0)}(s)|^{\theta m},
		\end{split}
	\end{align*}
	where the second term on the right-hand side can be estimated by Lemma~\ref{sec3:lem:2}. For the first term on the right-hand side we may apply \eqref{sec3:34} and obtain
	\begin{align*}
		\begin{split}
			&\fint_{I_{\ell}(t_0)}|(u\eta)_{B_{R}(x_0)}(t)-u_{B_{R}(x_0)}(t)|^{\theta m}\,dt\\
			&\qquad=\fint_{I_{\ell}(t_0)}\left|\fint_{B_{R}(x_0)} (u(x,t)-u_{B_{R}(x_0)}(t))\eta(x)\,dx\right|^{\theta m}\,dt\\
			&\qquad\le c \fint_{I_{\ell}(t_0)}\left(\fint_{B_{R}(x_0)} |u(x,t)-u_{B_{R}(x_0)}(t)|\,dx\right)^{\theta m}\,dt.
		\end{split}
	\end{align*}
	Therefore, using the Poincar\'e inequality in the spatial direction and H\"older's inequality, we have
	\begin{align*}
		\fint_{I_{\ell}(t_0)}|(u\eta)_{B_{R}(x_0)}(t)-u_{B_{R}(x_0)}(t)|^{\theta m}\,dt\le cR^{\theta m}\fiint_{Q_{R,\ell}(z_0)}|\na u|^{\theta m}\,dz.
	\end{align*}
	This completes the proof.
\end{proof}

\section{Parabolic Sobolev-Poincar\'e inequalities}
This section provides a parabolic Sobolev-Poincar\'e  inequality by adapting techniques in \cite{MR1749438} to the double-phase case. 
Throughout this section, let $z_0=(x_0,t_0)\in\Om_T$, with $x_0\in\Omega$ and $t_0\in(0,T)$, be a Lebesgue point of $|\na u(z)|^p+a(z)|\na u(z)|^q$ satisfying
\begin{align}\label{sec4:1}
	|\na u(z_0)|^p+a(z_0)|\na u(z_0)|^q>\La
\end{align}
for some 
$\La>1+\|a\|_{L^\infty(\Omega_T)}$.
Recall that $H(z,s):\Omega_T\times \RR^{+ }\longrightarrow \RR^+$,
$H(z,s)=s^p+a(z)s^q$.
For a fixed point $z_0$, we denote
\begin{align}\label{sec2:1}
	H_{z_0}(s)=s^p+a(z_0)s^q.
\end{align}
Note that $H_{z_0}(s)$ is strictly increasing and continuous with
\begin{align*}
	\lim_{s\to0^+}H_{z_0}(s)=0
	\quad\text{and}\quad
	\lim_{s\to\infty}H_{z_0}(s)=\infty.
\end{align*} 
By the intermediate value theorem for continuous functions, there exists $\la=\la(z_0)>1$ such that
\begin{align}\label{sec4:2}
	\La=\la^p+a(z_0)\la^q=H_{z_0}(\la).
\end{align}
Let
\begin{align*}
 M_1=\frac{1}{2|B_1|}\iint_{\Omega_T}\left(H(z,|\na u|)+H(z,|F|)\right)\,dz.
\end{align*}
The parameter $K=K(n,\alpha,[a]_{\alpha},M_1)>1$ will be determined later in \eqref{sec6:1}. 

The $p$-intrinsic cylinders and the $(p,q)$-intrinsic cylinders are considered separately in the argument.
In the $p$-intrinsic case we assume that
\begin{equation}\label{sec4:3}
		K\la^{p}\ge a(z_0)\la^q 
		\quad\text{and}\quad
		\fiint_{Q_{4\rho}^\la(z_0)}\left(H(z,|\na u|)+H(z,|F|)\right)\,dz< \la^p,
\end{equation}
where
\begin{equation}\label{sec4:31}
	Q_{4\rho}^\la(z_0)=
	B_{4\rho}(x_0)\times I_{4\rho}^\la(t_0),
 \quad I_{4\rho}^\la(t_0)=(t_0-\la^{2-p}(4\rho)^2,t_0+\la^{2-p}(4\rho)^2),
 \end{equation}
is a $p$-intrinsic cylinder.
In the $(p,q)$-intrinsic case we assume that
\begin{equation}\label{sec4:4}
	\begin{gathered}
		K\la^{p}\le a(z_0)\la^q,\quad
		\frac{a(z_0)}{2}\le a(z)\le 2a(z_0)\text{ for every } z\in G_{4\rho}^\la(z_0)
		\quad\text{and}\quad\\
		\fiint_{G_{4\rho}^\la(z_0)}\left(H(z,|\na u|)+H(z,|F|)\right)\,dz<H_{z_0}(\la),
	\end{gathered}
\end{equation}
where 
\begin{equation}\label{sec4:32}
G_{4\rho}^\la(z_0)=B_{4\rho}(x_0)\times J_{4\rho}^\la(t_0),\quad J_{4\rho}^\la(t_0)=\left(t_0-\tfrac{\la^2}{H_{z_0}(\la)}(4\rho)^2,t_0+\tfrac{\la^2}{H_{z_0}(\la)}(4\rho)^2\right),
\end{equation}
is a $(p,q)$-intrinsic cylinder.

\subsection{The $p$-intrinsic case}
In this case we consider estimates in $p$-intrinsic cylinders as in \eqref{sec4:31} and assume that \eqref{sec4:3} holds.
We begin by estimating the last term in Lemma~\ref{sec3:lem:3}.
\begin{lemma}\label{sec4:lem:1}
	Let $u$ be a weak solution to \eqref{sec1:1}. Then, for $s\in[2\rho,4\rho]$ and $\theta\in((q-1)/p,1]$, 
	there exists a constant $c=c(n,p,q,\alpha,L,[a]_{\alpha},M_1)$ such that
	\begin{align*}
		\begin{split}
			&\fiint_{Q_{s}^\la(z_0)}(|\na u|^{p-1}+a(z)|\na u|^{q-1}+|F|^{p-1}+a(z)|F|^{q-1})\,dz\\
			&\qquad\le c\fiint_{Q_{s}^\la(z_0)}(|\na u|+|F|)^{p-1}\,dz
			+c\la^{-1+\frac{p}{q}}\fiint_{Q_{s}^\la(z_0)}a(z)^\frac{q-1}{q}(|\na u|+|F|)^{q-1}\,dz\\
		&\qquad\qquad+c\la^{\frac{\alpha p}{n+2}}\left(\fiint_{Q_{s}^\la(z_0)}(|\na u|+|F|)^{\theta p}\,dz\right)^{\frac{1}{\theta}\left(\frac{p-1}{p}-\frac{\alpha}{n+2}\right)},
		\end{split}
	\end{align*}
 whenever $Q_{4\rho}^{\la}(z_0)\subset\Omega_T$ satisfies \eqref{sec4:3}.
\end{lemma}

\begin{proof}
	It follows from \eqref{sec1:4} that $q-1<p$. By \eqref{sec1:4} there exists a  constant $c=c([a]_{\alpha})$ such that
	\begin{equation}\label{sec4:41}
		\begin{split}
			&\fiint_{Q_{s}^\la(z_0)}(|\na u|^{p-1}+a(z)|\na u|^{q-1}+|F|^{p-1}+a(z)|F|^{q-1})\,dz\\
			&\qquad\le c\fiint_{Q_{s}^\la(z_0)}(|\na u|^{p-1}+|F|^{p-1})\,dz+c\fiint_{Q_{s}^\la(z_0)}\inf_{w\in Q_{s}^\la(z_0)}a(w)(|\na u|+|F|)^{q-1}\,dz\\
			&\qquad\qquad+cs^\alpha\fiint_{Q_{s}^\la(z_0)}(|\na u|^{q-1}+|F|^{q-1})\,dz.
		\end{split}
	\end{equation}
	We apply the first condition in \eqref{sec4:3} to estimate the second term on the right-hand side of \eqref{sec4:41} and obtain
	\begin{align*}
		\begin{split}
			&\fiint_{Q_{s}^\la(z_0)}\inf_{w\in Q_{s}^\la(z_0)}a(w)(|\na u|+|F|)^{q-1}\,dz\\
			&\qquad\le K^\frac{1}{q}\la^{-1+\frac{p}{q}}\fiint_{Q_{s}^\la(z_0)}\inf_{w\in Q_{s}^\la(z_0)}a(w)^\frac{q-1}{q}(|\na u|+|F|)^{q-1}\,dz\\
			&\qquad\le K^\frac{1}{q}\la^{-1+\frac{p}{q}}\fiint_{Q_{s}^\la(z_0)}a(z)^\frac{q-1}{q}(|\na u|+|F|)^{q-1}\,dz.
		\end{split}
	\end{align*}
	In order to estimate the last term on the right-hand side of \eqref{sec4:41}, we recall that $|Q_{s}^\la(z_0)|=c(n)s^{n+2}\la^{2-p}$. H\"older's inequality gives
	\begin{align*}
		\begin{split}
			&s^\alpha\fiint_{Q_{s}^\la(z_0)}|\na u|^{q-1}\,dz\le s^\alpha\left(\fiint_{Q_{s}^\la(z_0)}|\na u|^{\theta p}\,dz\right)^{\frac{1}{\theta}\frac{q-1}{p}}\\
			&\qquad\le s^\alpha\left(\fiint_{Q_{s}^\la(z_0)}|\na u|^{p}\,dz\right)^{\frac{\gamma}{p}}\left(\fiint_{Q_{s}^\la(z_0)}|\na u|^{\theta p}\,dz\right)^{\frac{1}{\theta}\frac{q-1-\gamma}{p}}\\
			&\qquad\le c s^{\alpha-\frac{(n+2)\gamma}{p}}\left(\iint_{Q_{s}^\la(z_0)}|\na u|^p\,dz\right)^\frac{\gamma}{p}\la^{\frac{(p-2)\gamma}{p}}\left(\fiint_{Q_{s}^\la(z_0)}|\na u|^{\theta p}\,dz\right)^{\frac{1}{\theta}\frac{q-1-\gamma}{p}},
		\end{split}
	\end{align*}
	where $c=c(n)$, $\gamma=\alpha p/(n+2)$ and $\theta\in((q-1)/p,1]$.
	We have $\gamma\in(0,p-1)$, since
	\begin{align*}
		 1<\frac{2(n+1)}{n+2}\Longrightarrow
		 1<\frac{(n+1)p}{n+2}\Longrightarrow
		 \gamma=\frac{\alpha p}{n+2}<p-1.
 \end{align*}
	It follows from the second condition in \eqref{sec4:3}, $\la\ge1$ and the first condition in \eqref{sec1:4} that
	\begin{align*}
		\left(\fiint_{Q_s^\la(z_0)}|\na u|^{\theta p}\,dz\right)^{\frac{1}{\theta}\frac{q-1-\gamma}{p}}
		&\le c\la^{q-p}\left(\fiint_{Q_s^\la(z_0)}|\na u|^{\theta p}\,dz\right)^{\frac{1}{\theta}\frac{p-1-\gamma}{p}}\\
		&\le c\la^{\frac{2\alpha}{n+2}}\left(\fiint_{Q_s^\la(z_0)}|\na u|^{\theta p}\,dz\right)^{\frac{1}{\theta}\frac{p-1-\gamma}{p}},
	\end{align*}
	where $c=c(n,p,q,\alpha)$. Therefore, we obtain
 \begin{align*}
     s^\alpha\fiint_{Q_{s}^\la(z_0)}|\na u|^{q-1}\,dz\le c\la^{\frac{\alpha p}{n+2}}\left(\fiint_{Q_s^\la(z_0)}|\na u|^{\theta p}\,dz\right)^{\frac{1}{\theta}\left(\frac{p-1}{p}-\frac{\alpha}{n+2}\right)}
 \end{align*}
where $c=c(n,p,q,\alpha,M_1)$. Similarly, replacing $|\na u|$ by $|F|$ in the above argument, we have
\begin{align*}
    s^\alpha\fiint_{Q_{s}^\la(z_0)}|F|^{q-1}\,dz\le c\la^{\frac{\alpha p}{n+2}}\left(\fiint_{Q_s^\la(z_0)}|F|^{\theta p}\,dz\right)^{\frac{1}{\theta}\left(\frac{p-1}{p}-\frac{\alpha}{n+2}\right)}
\end{align*}
This completes the proof.
\end{proof}
Next we provide a $p$-intrinsic parabolic Poincar\'e inequality.
\begin{lemma}\label{sec4:lem:2}
	Let $u$ be a weak solution to \eqref{sec1:1}. Then, for $s\in[2\rho,4\rho]$ and $\theta\in((q-1)/p,1]$, 
	there exists a constant $c=c(n,N,p,q,\alpha,L,[a]_{\alpha},M_1)$ such that 
	\begin{equation*}
		\begin{split}
			&\fiint_{Q_{s}^\la(z_0)}\frac{|u-u_{Q_{s}^\la(z_0)}|^{\theta p}}{s^{\theta p}}\,dz
			\le c\fiint_{Q_{s}^\la(z_0)}H(z,|\na u|)^\theta\,dz\\
			&\qquad+c\la^{\left(2-p+\frac{\alpha p}{n+2}\right)\theta p}\left(\fiint_{Q_{s}^\la(z_0)}(|\na u|+|F|)^{\theta p}\,dz\right)^{p-1-\frac{\alpha p}{n+2}}
			+c\left(\fiint_{Q_{s}^\la(z_0)}H(z,|F|)\,dz\right)^{\theta},
		\end{split}
	\end{equation*}
whenever $Q_{4\rho}^{\la}(z_0)\subset\Omega_T$ satisfies \eqref{sec4:3}.
\end{lemma}

\begin{proof}
	By Lemma~\ref{sec3:lem:3} and Lemma~\ref{sec4:lem:1}, there exists a constant $c=c(n,N,p,q,\alpha,L,[a]_{\alpha},M_1)$ such that
	\begin{equation}\label{sec4:42}
		\begin{split}
			&\fiint_{Q_{s}^\la(z_0)}\frac{|u-u_{Q_{s}^\la(z_0)}|^{\theta p}}{s^{\theta p}}\,dz
			\le c\fiint_{Q_{s}^\la(z_0)}|\na u|^{\theta p}\,dz\\
			&\qquad+c\left(\la^{2-p}\fiint_{Q_{s}^\la(z_0)}(|\na u|+|F|)^{p-1}\,dz\right)^{\theta p}\\
			&\qquad+c\left(\la^{1-p+\frac{p}{q}}\fiint_{Q_{s}^\la(z_0)}a(z)^\frac{q-1}{q}(|\na u|+|F|)^{q-1}\,dz\right)^{\theta p}\\
			&\qquad+c\la^{\left(2-p+\frac{\alpha p}{n+2}\right)\theta p}\left(\fiint_{Q_{s}^\la(z_0)}(|\na u|+|F|)^{\theta p}\,dz\right)^{p-1-\frac{\alpha p}{n+2}}.
		\end{split}
	\end{equation}
	To estimate the second term on the right-hand side of \eqref{sec4:42}, we use the second condition in \eqref{sec4:3} and obtain
	\begin{align*}
		\la^{(2-p)\theta p}\left(\fiint_{Q_{s}^\la(z_0)}(|\na u|+|F|)^{\theta p}\,dz\right)^{p-1}\le c\fiint_{Q_{s}^\la(z_0)}(|\na u|+|F|)^{\theta p}\,dz,
	\end{align*}
where $c=c(n,p)$.
Similarly, the third term on the right-hand side of \eqref{sec4:42} is estimated as 
\begin{align*}
		\la^{\left(1-p+\frac{p}{q}\right)\theta p}\left(\fiint_{Q_{s}^\la(z_0)}a(z)^{\theta }(|\na u|+|F|)^{\theta q}\,dz\right)^\frac{p(q-1)}{q}
		\le c\fiint_{Q_{s}^\la(z_0)}a(z)^{\theta }(|\na u|+|F|)^{\theta q}\,dz,
\end{align*}
where $c=c(n,p,q)$.
The conclusion follows from H\"older's inequality.
\end{proof}

\begin{lemma}\label{sec4:lem:3}
	Let $u$ be a weak solution to \eqref{sec1:1}. Then for $Q_{4\rho}^{\la}(z_0)\subset\Omega_T$ satisfying \eqref{sec4:3}, $s\in[2\rho,4\rho]$ and $\theta\in((q-1)/p,1]$, 
	there exists a constant $c=c(n,N,p,q,\alpha,L,[a]_{\alpha},M_1)$ such that 
	\begin{equation*}
		\begin{split}
			&\fiint_{Q_{s}^\la(z_0)}\inf_{w\in Q_{s}^\la(z_0)}a(w)^{\theta}\frac{|u-u_{Q_{s}^\la(z_0)}|^{\theta q}}{s^{\theta q}}\,dz
			\le c\fiint_{Q_{s}^\la(z_0)}H(z,|\na u|)^\theta\,dz\\
			&\qquad+c\la^{\left(2-p+\frac{\alpha p}{n+2}\right)\theta p}\left(\fiint_{Q_{s}^\la(z_0)}(|\na u|+|F|)^{\theta p}\,dz\right)^{p-1-\frac{\alpha p}{n+2}}
			+c\left(\fiint_{Q_{s}^\la(z_0)}H(z,|F|)\,dz\right)^{\theta}.
		\end{split}
	\end{equation*}
	
\end{lemma}

\begin{proof}
		By Lemma~\ref{sec3:lem:3} and Lemma~\ref{sec4:lem:1}, there exists a constant $c=c(n,N,p,q,\alpha,L,[a]_{\alpha},M_1)$ such that
	\begin{equation}\label{sec4:43}
		\begin{split}
			&\fiint_{Q_{s}^\la(z_0)}\inf_{w\in Q_{s}^{\la}(z_0)}a(w)^{\theta}\frac{|u-u_{Q_{s}^\la(z_0)}|^{\theta q}}{s^{\theta q}}\,dz 
			\le c\fiint_{Q_{s}^\la(z_0)}\inf_{w\in Q_{s}^\la(z_0)}a(w)^{\theta}|\na u|^{\theta q}\,dz\\
			&\qquad+ c\inf_{w\in Q_{s}^{\la}(z_0)}a(w)^{\theta}\left(\la^{2-p}\fiint_{Q_{s}^\la(z_0)}(|\na u|+|F|)^{p-1}\,dz\right)^{\theta q}\\
			&\qquad+c\inf_{w\in Q_{s}^{\la}(z_0)}a(w)^{\theta}\left(\la^{1-p+\frac{p}{q}}\fiint_{Q_{s}^\la(z_0)}a(z)^\frac{q-1}{q}(|\na u|+|F|)^{q-1}\,dz\right)^{\theta q}\\
			&\qquad+c\inf_{w\in Q_{s}^\la(z_0)}a(w)^{\theta}\la^{\left(2-p+\frac{\alpha p}{n+2}\right)\theta q}\left(\fiint_{Q_{s}^\la(z_0)}(|\na u|+|F|)^{\theta p}\,dz\right)^{q\left(\frac{p-1}{p}-\frac{\alpha}{n+2}\right)}.
		\end{split}
	\end{equation}
	By \eqref{sec4:3} for the second term on the right-hand side of \eqref{sec4:43}, we obtain
\begin{align*}
	\begin{split}
		&\inf_{w\in Q_{s}^\la(z_0)}a(w)^{\theta}\la^{(2-p)\theta q}\left(\fiint_{Q_{s}^\la(z_0)}(|\na u|+|F|)^{\theta p}\,dz\right)^{\frac{(p-1)q}{p}}\\
		&\qquad\le K^\theta\la^{(p-q)\theta}\la^{(2-p)\theta q+(p-1)q\theta-\theta p}\fiint_{Q_{s}^\la(z_0)}(|\na u|+|F|)^{\theta p}\,dz\\
		&\qquad\le c(n,p,\alpha,[a]_{\alpha},M_1)\fiint_{Q_{s}^\la(z_0)}(|\na u|+|F|)^{\theta p}\,dz.
	\end{split}
\end{align*}
Similarly, the third and the fourth terms on the right-hand side  of \eqref{sec4:43} can be estimated as
\begin{equation*}
	\begin{split}
		&\inf_{w\in Q_{s}^\la(z_0)}a(w)^{\theta}\la^{(p+q-pq)\theta }\left(\fiint_{Q_{s}^\la(z_0)}(a(z)(|\na u|^q+|F|^q))^{\theta}\,dz\right)^{q-1}\\
		&\qquad\le c\fiint_{Q_{s}^\la(z_0)}(a(z)(|\na u|^q+|F|^q))^{\theta}\,dz
	\end{split}
\end{equation*}
and
\begin{equation*}
	\begin{split}
		&\inf_{w\in Q_{s}^\la(z_0)}a(w)^{\theta}\la^{\left(2-p+\frac{\alpha p}{n+2}\right)\theta q}\left(\fiint_{Q_{s}^\la(z_0)}(|\na u|+|F|)^{\theta p}\,dz\right)^{q\left(\frac{p-1}{p}-\frac{\alpha }{n+2}\right)}\\
		&\qquad\le c\la^{\left(2-p+\frac{\alpha p}{n+2}\right)\theta p}\left(\fiint_{Q_{s}^\la(z_0)}(|\na u|+|F|)^{\theta p}\,dz\right)^{p-1-\frac{\alpha p}{n+2}}.
	\end{split}
\end{equation*}
	The conclusion follows from H\"older's inequality.
\end{proof}

\subsection{The $(p,q)$-intrinsic case}
In this case we consider estimates in $(p,q)$-intrinsic cylinders as in \eqref{sec4:32} and assume that \eqref{sec4:4} holds.
The second and third conditions in \eqref{sec4:4} imply
\begin{align*}
	\fiint_{G_{4\rho}^{\la}(z_0)}\left(H_{z_0}(|\na u|)+H_{z_0}(|F|)\right)\,dz<4a(z_0)\la^q.
\end{align*}
It follows that
\begin{align}\label{sec4:5}
	\fiint_{G_{4\rho}^\la(z_0)}(|\na u|^q+|F|^q)\,dz<4\la^q.
\end{align}

Next we discuss a  $(p,q)$-intrinsic parabolic Poincar\'e inequality.
\begin{lemma}\label{sec4:lem:5}
	Let $u$ be a weak solution to \eqref{sec1:1}. Then, for $\theta\in((q-1)/p,1]$ and $s\in[2\rho,4\rho]$, 
	there exists a constant $c=c(n,N,p,q,L)$ such that 
	\begin{align*}
	 \fiint_{G_{s}^\la(z_0)}H^{\theta}_{z_0}\left(\frac{|u-u_{G_{s}^\la(z_0)}|}{s}\right)\,dz
			\le c\fiint_{G_{s}^\la(z_0)}H_{z_0}^{\theta}(|\na u|)\,dz
			+c\left(\fiint_{G_s^\la(z_0)}H_{z_0}(|F|)\,dz\right)^\theta,
	\end{align*}
  whenever $G_{4\rho}^{\la}(z_0)\subset\Omega_T$ satisfies \eqref{sec4:4}.
\end{lemma}

\begin{proof}
	Note that
	\begin{align*}
 &\fiint_{G_{s}^\la(z_0)}H^{\theta}_{z_0}\left(\frac{|u-u_{G_{s}^\la(z_0)}|}{s}\right)\,dz\\
			&\qquad\le 2\fiint_{G_{s}^\la(z_0)}\left(\frac{|u-u_{G_{s}^\la(z_0)}|^{\theta p}}{s^{\theta p}}+a(z_0)^\theta\frac{|u-u_{G_{s}^\la(z_0)}|^{\theta q}}{s^{\theta q}}\right)\,dz.
	\end{align*}
    By Lemma~\ref{sec3:lem:3}, there exists a constant $c=c(n,N,p,q,L)$ such that
	\begin{equation}\label{sec4:51}
		\begin{split}		&\fiint_{G_{s}^\la(z_0)}H^{\theta}_{z_0}\left(\frac{|u-u_{G_{s}^\la(z_0)}|}{s}\right)\,dz\\
			&\qquad\le c\fiint_{G_{s}^\la(z_0)}H^{\theta}_{z_0}(|\na u|)\,dz
			+c H_{z_0}^\theta\left(\frac{\la}{H'_{z_0}(\la)}\fiint_{G_{s}^\la(z_0)}H'_{z_0}(|\na u|+|F|)\,dz\right).
		\end{split}
	\end{equation}
 Since
\begin{align*}
	H_{z_0}'(s)=ps^{p-1}+qa(z_0)s^{q-1}
\end{align*}
for every $s>0$, we have 
\begin{align}\label{sec4:6}
	sH_{z_0}'(s)\le q H_{z_0}(s) \le \tfrac{q}{p} sH_{z_0}'(s)
\end{align}
for every $s>0$.
	We estimate the last term on the right-hand side of \eqref{sec4:51} by \eqref{sec4:6} and obtain
	\begin{align*}
		\begin{split}
			&\frac{\la}{H'_{z_0}(\la)}\fiint_{G_s^\la(z_0)}H'_{z_0}(|\na u|)\,dz
			\le \frac{c\la}{\la^{p-1}+a(z_0)\la^{q-1}}\fiint_{G_{s}^\la(z_0)}(|\na u|^{p-1}+a(z_0)|\na u|^{q-1})\,dz\\
			&\qquad\le c\la\left(\frac{1}{\la^{p-1}} \fiint_{G_{s}^\la(z_0)}|\na u|^{p-1}\,dz+\frac{1}{\la^{q-1}}\fiint_{G_{s}^\la(z_0)}|\na u|^{q-1}\,dz\right).
		\end{split}
	\end{align*}
    By the same argument as above for $H_{z_0}'(|\na u|+|F|)$, it follows from \eqref{sec4:5} that
    \begin{align*}
    	\frac{\la}{H'_{z_0}(\la)}\fiint_{G_s^\la(z_0)}H'_{z_0}(|\na u|+|F|)\,dz\le c\la^{2-p}\left(\fiint_{G_{s}^\la(z_0)}(|\na u|+|F|)^{q-1}\,dz\right)^\frac{p-1}{q-1},
    \end{align*}
    where $c=c(n,p,q)$. Therefore, we obtain
	\begin{equation}\label{sec4:52}
		\begin{split}
			&H_{z_0}^{\theta}\left(\frac{\la}{H'_{z_0}(\la)}\fiint_{G_s^\la(z_0)}H'_{z_0}(|\na u|+|F|)\,dz\right)\\
			&\qquad\le c\left(\la^{(2-p) p}\left(\fiint_{G_s^\la(z_0)}(|\na u|+|F|)^{q-1}\,dz\right)^\frac{p(p-1)}{q-1}\right)^{\theta }\\
			&\qquad\qquad+c\left(a(z_0)\la^{(2-p) q}\left(\fiint_{G_s^\la(z_0)}(|\na u|+|F|)^{q-1}\,dz\right)^\frac{q(p-1)}{q-1}\right)^{\theta }.
		\end{split}
	\end{equation}
	In order to estimate the first term on the right-hand side of \eqref{sec4:52}, we apply \eqref{sec4:5}. Keeping in mind that $q-1<p$ and $p\ge2$, we have 
	\begin{align*}
		\begin{split}
			&\la^{(2-p)p}\left(\fiint_{G_s^\la(z_0)}(|\na u|+|F|)^{q-1}\,dz\right)^\frac{p(p-1)}{q-1}
			\le\la^{(2-p)p}\left(\fiint_{G_s^\la(z_0)}(|\na u|^{p}+|F|^{ p})^{\theta}\,dz\right)^{\frac{1}{\theta}(p-1)}\\
			&\qquad\qquad\qquad\le c\left(\fiint_{G_s^\la(z_0)}|\na u|^{\theta p}\,dz\right)^\frac{1}{\theta}+c\fiint_{G_{s}^\la(z_0)}|F|^p\,dz,
		\end{split}
	\end{align*}
 for any $\theta\in((q-1)/p,1]$ with $c=c(n,p)$.
	We estimate the last term on the right-hand side of \eqref{sec4:52} in a similar way. Then for any $\theta\in((q-1)/q,1]$, we have
	\begin{align*}
		\begin{split}
			&a(z_0)\la^{(2-p)q}\left(\fiint_{G_s^\la(z_0)}(|\na u|+|F|)^{q-1}\,dz\right)^\frac{q(p-1)}{q-1}\\
			&\qquad\le a(z_0)\la^{(2-p)q}\left(\fiint_{G_s^\la(z_0)}(|\na u|+|F|)^{\theta q}\,dz\right)^{\frac{1}{\theta}(p-1)}\\
			&\qquad\le c\left(\fiint_{G_{s}^\la(z_0)}a(z_0)^\theta|\na u|^{\theta q}\,dz\right)^\frac{1}{\theta}+c\fiint_{G_{s}^\la(z_0)}a(z_0)|F|^q\,dz,
		\end{split}
	\end{align*}
 where $c=c(n,p)$.
	Hence, we conclude that
	\begin{align*}
     &H_{z_0}^{\theta}\left(\frac{\la}{H'_{z_0}(\la)}\fiint_{G_s^\la(z_0)}H'_{z_0}(|\na u|+|F|)\,dz\right)\\
			 &\qquad\le c\fiint_{G_s^\la(z_0)}H_{z_0}^{\theta }(|\na u|)\,dz
  +\left(\fiint_{G_{s}^\la(z_0)}H_{z_0}(|F|)\,dz\right)^\theta,
	\end{align*}
	which completes the proof.
\end{proof}
Note that by replacing $H_{z_0}^{\theta }(s)$ with $s^{\theta p}$ in the proof of Lemma~\ref{sec4:lem:5}, we will also have the following result. All necessary calculations are already contained in the proof of the previous lemma. 
\begin{lemma}\label{sec4:lem:6}
	Let $u$ be a weak solution to \eqref{sec1:1}. Then, for $\theta\in((q-1)/p,1]$ and $s\in[2\rho,4\rho]$, there exists a constant $c=c(n,N,p,q,L)$ such that 
	\begin{align*}
			\fiint_{G_{s}^\la(z_0)}\left(\frac{|u-u_{G_{s}^\la(z_0)}|}{s}\right)^{\theta p}\,dz
			\le c\fiint_{G_{s}^\la(z_0)}|\na u|^{\theta p}\,dz+c\left(\fiint_{G_s^\la(z_0)}|F|^p\,dz\right)^\theta,
	\end{align*}
whenever $G_{4\rho}^{\la}(z_0)\subset\Omega_T$ satisfies \eqref{sec4:4}.
\end{lemma}

\section{Reverse H\"older inequalities}\label{section5}

In this section, we assume that $z_0$, $\Lambda$ and $\la=\la(z_0)\ge1$ satisfy \eqref{sec4:1}-\eqref{sec4:2}. Let 
\[
K=1+40[a]_\alpha M_1^\frac{\alpha}{n+2}\quad\text{and}\quad\cv=10K.
\]
The distribution sets are denoted as
\begin{equation}\label{sec2:2}
\Psi(\La)=\{ z\in \Omega_T: H(z,|\na u(z)|)>\La\} 
\end{equation}
and
\begin{equation}\label{sec2:3}
\Phi(\La)=\{ z\in \Omega_T: H(z,|F|)>\La\}.
\end{equation}
We consider the $p$-intrinsic and $(p,q)$-intrinsic cases separately. In the first case we assume that
\begin{equation}\label{sec5:1}
	\begin{gathered}
		K\la^{p}\ge a(z_0)\la^q,\\
	     \fiint_{Q_{s}^\la(z_0)}\left(H(z,|\na u|)+H(z,|F|)\right)\,dz< \la^p\text{ for every } s\in(\rho,2\cv\rho]
      \text{ and }\\
		\fiint_{Q_{\rho}^\la(z_0)}\left(H(z,|\na u|)+H(z,|F|)\right)\,dz= \la^p,
	\end{gathered}
\end{equation}
where the $p$-intrinsic cylinder is defined in \eqref{sec4:31}.
In the second case we assume that
\begin{equation}\label{sec5:2}
	\begin{gathered}
		K\la^{p}\le a(z_0)\la^q,\quad \frac{a(z_0)}{2}\le a(z)\le 2a(z_0)\text{ for every } z\in G_{4\rho}^\la(z_0),\\
		\fiint_{G_{s}^\la(z_0)}\left(H(z,|\na u|)+H(z,|F|)\right)\,dz< H_{z_0}(\la)\text{ for every } s\in(\rho,2\cv\rho]
  \text{ and }\\
		\fiint_{G_{\rho}^\la(z_0)}\left(H(z,|\na u|)+H(z,|F|)\right)\,dz=H_{z_0}(\la),
	\end{gathered}
\end{equation}
where the $(p,q)$-intrinsic cylinder is defined in \eqref{sec4:32}.
We discuss reverse H\"older inequalities in both cases separately.

The following auxiliary lemmas will be employed in the argument. The first lemma is a Gagliardo-Nirenberg inequality and the second one is a standard iteration lemma, see \cite[Lemma 8.3]{MR1962933}.
\begin{lemma}\label{sec2:lem:1}
	Let $B_{\rho}(x_0)\subset\RR^n$, $\sig,s,r\in[1,\infty)$ and $\vartheta\in(0,1)$ such that 
	\begin{align*}
		-\frac{n}{\sig}\le \vartheta\left(1-\frac{n}{s}\right)-(1-\vartheta)\frac{n}{r}.
	\end{align*}
	Then there exists a constant $c=c(n,\sig)$ such that
	\begin{align*}
		\fint_{B_{\rho}(x_0)}\frac{|v|^\sig}{\rho^\sig}\,dx
		\le c\left(\fint_{B_{\rho}(x_0)}\left(\frac{|v|^s}{\rho^s}+|\na v|^s\right)\,dx\right)^\frac{\vartheta \sig}{s}\left(\fint_{B_{\rho}(x_0)}\frac{|v|^r}{\rho^r}\,dx\right)^\frac{(1-\vartheta)\sig}{r}
	\end{align*}
  for every $v\in W^{1,s}(B_{\rho}(x_0))$.
\end{lemma}
\begin{lemma}\label{sec2:lem:2}
	Let $0<r<R<\infty$ and $h:[r,R]\longrightarrow\RR$ be a non-negative and bounded function. Suppose there exist $\vartheta\in(0,1)$, $A,B\ge0$ and $\gamma>0$ such that
	\begin{align*}
		h(r_1)\le \vartheta h(r_2)+\frac{A}{(r_2-r_1)^\gamma}+B
		\quad\text{for all}\quad
		0<r\le r_1<r_2\le R.
	\end{align*}
	Then there exists a constant $c=c(\vartheta,\gamma)$ such that
	\begin{align*}
		h(r)\le c\left(\frac{A}{(R-r)^\gamma}+B\right).
	\end{align*}
\end{lemma}

\subsection{The $p$-intrinsic case} In this case we consider estimates in $p$-intrinsic cylinders as in \eqref{sec4:31} and assume that \eqref{sec5:1} holds.
We denote
	\begin{align*}
		S(u,Q_{\rho}^\la(z_0))=\sup_{I_{\rho}^\la(t_0)}\fint_{B_{\rho}(x_0)}\frac{|u-u_{Q_{\rho}^\la(z_0)}|^2}{\rho^2}\,dx
	\end{align*}	
and
$M_2=\|u\|_{L^\infty(0,T;L^2(\Omega))}$.

\begin{lemma}\label{sec5:lem:1}
	Let $u$ be a weak solution to \eqref{sec1:1}. Then there exists a constant $c=c(\data)$ such that 
	\begin{align*}
		S(u,Q_{2\rho}^\la(z_0))=\sup_{I_{2\rho}^\la(t_0)}\fint_{B_{2\rho}(x_0)}\frac{|u-u_{Q_{2\rho}^\la(z_0)}|^2}{(2\rho)^2}\,dx\le c\la^2,
	\end{align*}	
 whenever $Q_{2\cv\rho}^{\la}(z_0)\subset\Omega_T$ satisfies \eqref{sec5:1}.
\end{lemma}

\begin{proof}
	Let $2\rho\le \rho_1<\rho_2\le 4\rho$. By Lemma~\ref{sec3:lem:1} there exists a constant $c=c(n,p,q,\nu,L)$ such that
	\begin{align}\label{sec5:3}
		\begin{split}
			&\la^{p-2}S(u,Q_{\rho_1}^\la(z_0))
			\le \frac{c\rho_2^q}{(\rho_2-\rho_1)^q}\fiint_{Q_{\rho_2}^\la(z_0)}\left(\frac{|u-u_{Q_{\rho_2}^\la(z_0)}|^p}{\rho_2^p}+a(z)\frac{|u-u_{Q_{\rho_2}^\la(z_0)}|^q}{\rho_2^q}\right)\,dz\\
			&\qquad+\frac{c \rho_2^2\la^{p-2}}{(\rho_2-\rho_1)^2}\fiint_{Q_{\rho_2}^\la(z_0)}\frac{|u-u_{Q_{\rho_2}^\la(z_0)}|^2}{\rho_2^2}\,dz+c\fiint_{Q_{\rho_2}^\la(z_0)}H(z,|F|)\,dz.
		\end{split}
	\end{align}
   
   We estimate the first term on the right-hand side of \eqref{sec5:3}. By Lemma~\ref{sec4:lem:2}  with $\theta=1$ and the second condition in \eqref{sec5:1}, we obtain 
	\begin{align}\label{sec5:4}
		\fiint_{Q_{\rho_2}^\la(z_0)}\frac{|u-u_{Q_{\rho_2}^\la(z_0)}|^{ p}}{\rho_2^{p}}\,dz\le c\la^{p},
	\end{align}
 where $c=c(n,N,p,q,\alpha,L,[a]_{\alpha},M_1)$.
On the other hand, we observe that
\begin{align*}
	\begin{split}
		&\fiint_{Q_{\rho_2}^\la(z_0)}a(z)\frac{|u-u_{Q_{\rho_2}^\la(z_0)}|^q}{\rho_2^q}\,dz\\
		&\qquad\le \fiint_{Q_{\rho_2}^\la(z_0)}\inf_{w\in 
			Q_{\rho_2}^\la(z_0)}a(w)\frac{|u-u_{Q_{\rho_2}^\la(z_0)}|^q}{\rho_2^q}\,dz
		+[a]_\alpha\rho_2^\alpha\fiint_{Q_{\rho_2}^\la(z_0)}\frac{|u-u_{Q_{\rho_2}^\la(z_0)}|^q}{\rho_2^q}\,dz.
	\end{split}
\end{align*}
	By  Lemma~\ref{sec4:lem:3} with $\theta=1$ and \eqref{sec5:1}, we obtain
	\begin{align*}
		\fiint_{Q_{\rho_2}^\la(z_0)}\inf_{w\in Q_{\rho_2}^\la(z_0)}a(w) \frac{|u-u_{Q_{\rho_2}^\la(z_0)}|^{ q}}{\rho_2^{q}}\,dz\le c\la^p,
	\end{align*}
 where $c=c(n,N,p,q,\alpha,L,[a]_{\alpha},M_1)$.
	On the other hand, by Lemma~\ref{sec2:lem:1} with $\sig=q$, $s=p$, $r=2$ and $\vartheta=\tfrac{p}{q}$, we obtain
	\begin{align*}
		\begin{split}
		&\rho_2^\alpha\fiint_{Q_{\rho_2}^\la(z_0)}\frac{|u-u_{Q_{\rho_2}^\la(z_0)}|^q}{\rho_2^q}\,dz\\
			&\qquad\le c\rho_2^\alpha\fiint_{Q_{\rho_2}^\la(z_0)}\left(\frac{|u-u_{Q_{\rho_2}^\la(z_0)}|^{p}}{\rho_2^{ p}}+|\na u|^{ p}\right)\,dz \left(S(u,Q_{\rho_2}^\la(z_0))\right)^\frac{q- p}{2}
		\end{split}
	\end{align*}
 where $c=c(n,q)$.
    We observe that
    \begin{align*}
    	\begin{split}
    		&\rho_2^\alpha \left(\sup_{I_{\rho_2}^\la(t_0)}\fint_{B_{\rho_2}(x_0)} \frac{|u-u_{Q_{\rho_2}^\la(z_0)}|^2}{\rho_2^2}\,dx\right)^\frac{q- p}{2}\le \rho_2^{\alpha}\left(2^2\sup_{I_{\rho_2}^\la(t_0)}\fint_{B_{\rho_2}(x_0)} \frac{|u|^2}{\rho_2^2}\,dx\right)^\frac{q- p}{2}\\
    		&\qquad\qquad\qquad\le\rho^{\alpha-\frac{(q-p)(n+2)}{2}}\left(2^2\sup_{I_{\rho_2}^\la(t_0)}\int_{B_{\rho_2}(x_0)}|u|^2\,dx\right)^\frac{q-p}{2}\le c,
    	\end{split}
    \end{align*}
    where $c=c(n,N,p,q,\alpha,\diam(\Om),M_2)$.
    Furthermore, by \eqref{sec5:4} we have
    \begin{align*}
    	\rho_2^\alpha\fiint_{Q_{\rho_2}^\la(z_0)}\frac{|u-u_{Q_{\rho_2}^\la(z_0)}|^q}{\rho_2^q}\,dz\le c\la^p,
    \end{align*}
    where $c=c(n,N,p,q,\alpha,L,[a]_\alpha,\diam(\Om),M_1,M_2)$.
    
   For the second term on the right-hand side of~\eqref{sec5:3}, the Poincar\'e inequality implies that
	\begin{align*}
		\begin{split}
			&\fiint_{Q_{\rho_2}^\la(z_0)}\frac{|u-u_{Q_{\rho_2}^\la(z_0)}|^2}{\rho_2^2}\,dz\\
			&\qquad=\fint_{I_{\rho_2}^\la(t_0)}\left(\fint_{B_{\rho_2}(x_0)}\frac{|u-u_{Q_{\rho_2}^\la(z_0)}|^2}{\rho_2^2}\,dx\right)^\frac{1}{2}\left(\fint_{B_{\rho_2}(x_0)}\frac{|u-u_{Q_{\rho_2}^\la(z_0)}|^2}{\rho_2^2}\,dx\right)^\frac{1}{2}\,dt\\
			&\qquad\le c \fint_{I_{\rho_2}^\la(t_0)}\left(\fint_{B_{\rho_2}(x_0)}\left(\frac{|u-u_{Q_{\rho_2}^\la(z_0)}|^{ p}}{\rho_2^{ p}}+|\na u|^{ p}\right)\,dx\right)^\frac{1}{p}\,dt\left(S(u,Q_{\rho_2}^\la(z_0))\right)^\frac{1}{2},
		\end{split}
	\end{align*}
 where $c=c(n,N,p)$.
	By H\"older's inequality and \eqref{sec5:4}, we have
	\begin{align*}
		\fiint_{Q_{\rho_2}^\la(z_0)}\frac{|u-u_{Q_{\rho_2}^\la(z_0)}|^2}{\rho_2^2}\,dz
  \le c\lambda S(u,Q_{\rho_2}^\la(z_0))^\frac{1}{2},
	\end{align*}
	where $c=c(n,N,p,q,\alpha,L,[a]_{\alpha},M_1)$.
	
   For the last term on the right-hand side of \eqref{sec5:3}, by \eqref{sec5:1} we obtain
	\begin{align*}
		\fiint_{Q_{\rho_2}^\la(z_0)}H(z,|F|)\,dz\le \la^p.
	\end{align*}
	By combining all estimates above, we conclude from \eqref{sec5:3} that
	\begin{align*}
		\begin{split}
			S(u,Q_{\rho_1}^\la(z_0))\le c\frac{\rho_2^q}{(\rho_2-\rho_1)^q}\la^2+c\frac{ \rho_2^2}{(\rho_2-\rho_1)^2}\la\ S(u,Q_{\rho_2}^\la(z_0))^\frac{1}{2}.
		\end{split}
	\end{align*}
    Finally, we  apply Young's inequality to obtain
    \begin{align*}
    	\begin{split}
    		S(u,Q_{\rho_1}^\la(z_0))\le \frac{1}{2}S(u,Q_{\rho_2}^\la(z_0))+ c\left(\frac{\rho_2^q}{(\rho_2-\rho_1)^q}+\frac{ \rho_2^4}{(\rho_2-\rho_1)^4}\right)\la^2.
    	\end{split}
    \end{align*}
	The proof is concluded by an application of Lemma~\ref{sec2:lem:2}.	
\end{proof}

Next we prove an estimate for the first term on the right-hand side of the energy estimate in Lemma~\ref{sec3:lem:1} by using Lemma~\ref{sec2:lem:1}.
\begin{lemma}\label{sec5:lem:2}
	Let $u$ be a weak solution to \eqref{sec1:1}. Then there exist constants $c=c(\data)$ and $\theta_0=\theta_0(n,p,q)\in(0,1)$ such that for any $\theta\in(\theta_0,1)$,
	\begin{align*}
		\begin{split}
			&\fiint_{Q_{2\rho}^\la(z_0)}\left(\frac{|u-u_{Q_{2\rho}^\la(z_0)}|^p}{(2\rho)^p}+a(z)\frac{|u-u_{Q_{2\rho}^\la(z_0)}|^q}{(2\rho)^q}\right)\,dz\\
			&\qquad\le  c\fiint_{Q_{2\rho}^\la(z_0)}\left(\frac{|u-u_{Q_{2\rho}^\la(z_0)}|^{\theta p}}{(2\rho)^{\theta p}}+|\na u|^{\theta p}\right)\,dz\left(S(u,Q_{2\rho}^\la(z_0))\right)^\frac{(1-\theta )p}{2}\\
			&\qquad\qquad+c  \fiint_{Q_{2\rho}^\la(z_0)}\left(\inf_{w\in Q_{2\rho}^\la(z_0)}a(w)^{\theta}\frac{|u-u_{Q_{2\rho}^\la(z_0)}|^{\theta q}}{(2\rho)^{\theta q}}+\inf_{w\in Q_{2\rho}^\la(z_0)}a(w)^{\theta}|\na u|^{\theta q}\right)\,dz\\
			&\qquad\qquad\qquad\times\la^{(p-q)(1-\theta)}\left(S(u,Q_{2\rho}^\la(z_0))\right)^\frac{(1-\theta)q}{2},
		\end{split}
	\end{align*}
  whenever $Q_{2\cv\rho}^{\la}(z_0)\subset\Omega_T$ satisfies \eqref{sec5:1}.
\end{lemma}

\begin{proof}
	By \eqref{sec1:7} we obtain
	\begin{equation}\label{sec5:41}
		\begin{split}
			&\fiint_{Q_{2\rho}^\la(z_0)}\left(\frac{|u-u_{Q_{2\rho}^\la(z_0)}|^p}{(2\rho)^p}+a(z)\frac{|u-u_{Q_{2\rho}^\la(z_0)}|^q}{(2\rho)^q}\right)\,dz\\
			&\qquad\le\fiint_{Q_{2\rho}^\la(z_0)}\frac{|u-u_{Q_{2\rho}^\la(z_0)}|^p}{(2\rho)^p}\,dz+\fiint_{Q_{2\rho}^\la(z_0)}\inf_{w\in Q_s^\la(z_0)}a(w)\frac{|u-u_{Q_{2\rho}^\la(z_0)}|^q}{(2\rho)^q}\,dz\\
			&\qquad\qquad+[a]_{\alpha}(2\rho)^\alpha\fiint_{Q_{2\rho}^\la(z_0)}\frac{|u-u_{Q_{2\rho}^\la(z_0)}|^q}{(2\rho)^q}\,dz.
		\end{split}
	\end{equation}
	
	We begin with the first term on the right-hand side of \eqref{sec5:41}. By choosing $\sig=p$, $s=\theta p$ and $r=2$, we see that any $\theta\in(n/(n+2),1)$ satisfies the condition in Lemma~\ref{sec2:lem:1} as
	\begin{align*}
		-\frac{n}{p}\le\theta\left(1-\frac{n}{\theta p}\right)-(1-\theta)\frac{n}{2}\Longleftrightarrow \frac{n}{n+2}\le \theta.
	\end{align*}
	Thus, we obtain
	\begin{align*}
			&\fiint_{Q_{2\rho}^\la(z_0)}\frac{|u-u_{Q_{2\rho}^\la(z_0)}|^p}{(2\rho)^p}\,dz\\
			&\qquad\le c\fiint_{Q_{2\rho}^\la(z_0)}\left(\frac{|u-u_{Q_{2\rho}^\la(z_0)}|^{\theta p}}{(2\rho)^{\theta p}}+|\na u|^{\theta p}\right)\,dz\left(S(u,Q_{2\rho}^\la(z_0))\right)^{\frac{(1-\theta)p}{2}},
	\end{align*}
 where $c=c(n,p)$.

	For the second term on the right-hand side of \eqref{sec5:41}, we apply Lemma~\ref{sec2:lem:1} with $\sig=q$, $s=\theta q$ and $r=2$. For any $\theta\in(n/(n+2),1)$, we have
	\begin{align*}
		\begin{split}
			&\fiint_{Q_{2\rho}^\la(z_0)}\inf_{w\in Q_{2\rho}^\la(z_0)}a(w)\frac{|u-u_{Q_{2\rho}^\la(z_0)}|^q}{(2\rho)^q}\,dz\\
			&\qquad\le c\fiint_{Q_{2\rho}^\la(z_0)}\left(\inf_{w\in Q_{2\rho}^\la(z_0)}a(w)^{\theta}\frac{|u-u_{Q_{2\rho}^\la(z_0)}|^{\theta q}}{(2\rho)^{\theta q}}
			+\inf_{w\in Q_{2\rho}^\la(z_0)}a(w)^{\theta}|\na u|^{\theta q}\right)\,dz\\
			&\qquad\qquad\times \inf_{w\in Q_{2\rho}^\la(z_0)}a(w)^{1-\theta} \left(S(u,Q_{2\rho}^\la(z_0))\right)^\frac{(1-\theta)q}{2}
		\end{split}
	\end{align*}
 where $c=c(n,q)$.
	By using the first condition in \eqref{sec5:1}, we have
	\begin{align*}
		\begin{split}
	     &\fiint_{Q_{2\rho}^\la(z_0)}\inf_{w\in Q_{2\rho}^\la(z_0)}a(w)\frac{|u-u_{Q_{2\rho}^\la(z_0)}|^q}{(2\rho)^q}\,dz\\
        &\qquad\le c\fiint_{Q_{2\rho}^\la(z_0)}\left(\inf_{w\in Q_{2\rho}^\la(z_0)}a(w)^{\theta}\frac{|u-u_{Q_{2\rho}^\la(z_0)}|^{\theta q}}{(2\rho)^{\theta q}}
        +\inf_{w\in Q_{2\rho}^\la(z_0)}a(w)^{\theta}|\na u|^{\theta q}\right)\,dz\\
        &\qquad\qquad\times \la^{(p-q)(1-\theta)}S(u,Q_{2\rho}^\la(z_0))^\frac{(1-\theta)q}{2}.
		\end{split}
	\end{align*}
	
	Then we consider the last term on the right-hand side of \eqref{sec5:41}. We observe that
	\begin{align*}
		\frac{nq}{(n+2)p}\le \frac{n}{n+2}\left(1+\frac{2}{(n+2)p}\right)\le \frac{n}{n+2}\frac{n+3}{n+2}<1.
	\end{align*}
	Thus, by letting $\sig=q$, $s=\theta p$, $r = 2$ and $\vartheta=\theta p/q$, the assumptions in Lemma~\ref{sec2:lem:1} are satisfied for any $\theta\in(nq/((n+2)p),1)$, since
	\begin{align*}
		-\frac{n}{q}\le \frac{\theta p}{q}\left(1-\frac{n}{\theta p}\right)-\left(1-\frac{\theta p}{q}\right)\frac{n}{2}\Longleftrightarrow \frac{n q}{(n+2)p}\le \theta.
	\end{align*} 
	Therefore, we have
	\begin{align*}
		\begin{split}
			&(2\rho)^\alpha\fiint_{Q_{2\rho}^\la(z_0)}\frac{|u-u_{Q_{2\rho}^\la(z_0)}|^q}{(2\rho)^q}\,dz\\
			&\qquad\le c\fiint_{Q_{2\rho}^\la(z_0)}\left(\frac{|u-u_{Q_{2\rho}^\la(z_0)}|^{\theta p}}{(2\rho)^{\theta p}}+|\na u|^{\theta p}\right)\,dz\left(S(u,Q_{2\rho}^\la(z_0))\right)^{\frac{p(1-\theta)}{2}}\\
			&\qquad\qquad\times
			(2\rho)^\alpha\left(\sup_{I_{2\rho}^\la(t_0)}\fint_{B_{2\rho}(x_0)} \frac{|u-u_{Q_{2\rho}^\la(z_0)}|^2}{(2\rho)^2}\,dx\right)^\frac{q-p}{2},
		\end{split}
	\end{align*}
 where $c=c(n,q)$.
	Note that 
	\begin{equation*}
			(2\rho)^\alpha\left(\sup_{I_{2\rho}^\la(t_0)}\fint_{B_{2\rho}(x_0)} \frac{|u-u_{Q_{2\rho}^\la(z_0)}|^2}{(2\rho)^2}\,dx\right)^\frac{q- p}{2}
			\le(2\rho)^\alpha\left(4\sup_{I_{2\rho}^\la(t_0)}\fint_{B_{2\rho}(x_0)} \frac{|u|^2}{(2\rho)^2}\,dx\right)^\frac{q-p }{2}.
	\end{equation*}
Thus we obtain
	\begin{align*}
		\begin{split}
			&(2\rho)^\alpha\left(\sup_{I_{2\rho}^\la(t_0)}\fint_{B_{2\rho}(x_0)} \frac{|u-u_{Q_{2\rho}^\la(z_0)}|^2}{(2\rho)^2}\,dx\right)^\frac{q- p}{2}\\
			&\qquad\le c(2\rho)^{\alpha-\frac{(q-p)(n+2)}{2}}\left(\sup_{I_{2\rho}^\la(t_0)}\int_{B_{2\rho}(x_0)} |u|^2\,dx\right)^\frac{q-p }{2}\le c,
		\end{split}
	\end{align*}
 where $c=c(n,p,q,\alpha,\diam(\Om),M_2)$.
 The claim follows by combining the estimates above.
\end{proof}
At this stage, we have all the required tools to prove the reverse H\"older inequality when \eqref{sec5:1} holds true.
\begin{lemma}\label{sec5:lem:3}
	Let $u$ be a weak solution to \eqref{sec1:1}. Then there exist constants $c=c(\data)$ and $\theta_0=\theta_0(n,p,q)\in(0,1)$ such that for any $\theta\in(\theta_0,1)$,
	\begin{align*}
		\begin{split}
			&\fiint_{Q_{\rho}^\la(z_0)}\left(H(z,|\na u|)+H(z,|F|)\right)\,dz \\
			&\qquad\le c\left(\fiint_{Q_{2\rho}^\la(z_0)}H(z,|\na u|)^\theta\,dz\right)^\frac{1}{\theta}+c\fiint_{Q_{2\rho}^\la(z_0)}H(z,|F|)\,dz,
		\end{split}
	\end{align*}	
  whenever $Q_{2\cv\rho}^{\la}(z_0)\subset\Omega_T$ satisfies \eqref{sec5:1}.
\end{lemma}

\begin{proof}
	Lemma~\ref{sec3:lem:1} implies that
	\begin{align}\label{sec5:5}
		\begin{split}
			&\fiint_{Q_{\rho}^\la(z_0)}H(z,|\na u|)\,dz
			\le c\fiint_{Q_{2\rho}^\la(z_0)}\left(\frac{|u-u_{Q_{2\rho}^\la(z_0)}|^p}{(2\rho)^p}+a(z)\frac{|u-u_{Q_{2\rho}^\la(z_0)}|^q}{(2\rho)^q}\right)\,dz\\
			&\qquad\qquad+c\la^{p-2}\fiint_{Q_{2\rho}^\la(z_0)}\frac{|u-u_{Q_{2\rho}^\la(z_0)}|^2}{(2\rho)^2}\,dz+c\fiint_{Q_{2\rho}^\la(z_0)}H(z,|F|)\,dz,
		\end{split}
	\end{align}
	where $c=c(n,p,q,\nu,L)$.
	To estimate the first term on the right-hand side in \eqref{sec5:5}, we apply Lemma~\ref{sec5:lem:1} and Lemma~\ref{sec5:lem:2} to conclude that there exist $\theta_0=\theta_0(n,p,q)\in(0,1)$ and $c=c(\data)$ such that for any $\theta\in(\theta_0,1)$,
	\begin{align*}
		\begin{split}
			&\fiint_{Q_{2\rho}^\la(z_0)}\left(\frac{|u-u_{Q_{2\rho}^\la(z_0)}|^p}{(2\rho)^p}+a(z)\frac{|u-u_{Q_{2\rho}^\la(z_0)}|^q}{(2\rho)^q}\right)\,dz
			\le  c \la^{(1-\theta )p}\fiint_{Q_{2\rho}^\la(z_0)}H(z,|\na u|)^{\theta}\,dz\\
			&\qquad\qquad+c \la^{(1-\theta)p}\fiint_{Q_{2\rho}^\la(z_0)}\left(\frac{|u-u_{Q_{2\rho}^\la(z_0)}|^{\theta p}}{(2\rho)^{\theta p}}+\inf_{w\in Q_{2\rho}^\la(z_0)}a(w)^{\theta}\frac{|u-u_{Q_{2\rho}^\la(z_0)}|^{\theta q}}{(2\rho)^{\theta q}}\right)\,dz.
		\end{split}
	\end{align*}
	By Lemma~\ref{sec4:lem:2} and Lemma~\ref{sec4:lem:3} we obtain 
	\begin{align*}
		\begin{split}
			&\fiint_{Q_{2\rho}^\la(z_0)}\left(\frac{|u-u_{Q_{2\rho}^\la(z_0)}|^p}{(2\rho)^p}+a(z)\frac{|u-u_{Q_{2\rho}^\la(z_0)}|^q}{(2\rho)^q}\right)\,dz\\
			&\qquad\le  c \la^{(1-\theta )p}\fiint_{Q_{2\rho}^\la(z_0)}H(z,|\na u|)^{\theta}\,dz\\
			&\qquad\qquad+c  \la^{p-\left(p-1-\frac{\alpha p}{n+2}\right)\theta p}\left(\fiint_{Q_{2 \rho}^\la(z_0)}(|\na u|+|F|)^{\theta p}\,dz\right)^{p-1-\frac{\alpha p}{n+2}}\\
			&\qquad\qquad+c \la^{(1-\theta)p}\left(\fiint_{Q_{2 \rho}^\la(z_0)}H(z,|F|)\,dz\right)^{\theta}.
		\end{split}
	\end{align*}
    By recalling that $\tfrac{\alpha p}{n+2} < p-1$ and letting
    \begin{align*}
    	\beta=\min\left\{p-1-\frac{\alpha p}{n+2},\frac{1}{2}\right\},
    \end{align*}
    we have
    \begin{align*}
    	\begin{split}
    		&\fiint_{Q_{2\rho}^\la(z_0)}\left(\frac{|u-u_{Q_{2\rho}^\la(z_0)}|^p}{(2\rho)^p}+a(z)\frac{|u-u_{Q_{2\rho}^\la(z_0)}|^q}{(2\rho)^q}\right)\,dz\\
    		&\qquad\le  c \la^{p(1-\beta\theta)}\left(\fiint_{Q_{2\rho}^\la(z_0)}H(z,|\na u|)^{\theta}\,dz\right)^{\beta}
    		+c   \la^{(1-\beta \theta)p}\left(\fiint_{Q_{2\rho}^\la(z_0)}H(z,|F|)\,dz\right)^{\beta \theta}.
    	\end{split}
    \end{align*}

	To estimate the second term on the right-hand side of \eqref{sec5:5}, we apply the Poincar\'e inequality with $\theta\in(2n/((n+2)p),1)$ and Lemma~\ref{sec5:lem:1} to obtain
	\begin{align*}
		\begin{split}
			&\fiint_{Q_{2\rho}^\la(z_0)}\frac{|u-u_{Q_{2\rho}^\la(z_0)}|^2}{(2\rho)^2}\,dz\\
			&\qquad=\fint_{I_{2\rho}^\la(t_0)}\left(\fint_{B_{2\rho}(x_0)}\frac{|u-u_{Q_{2\rho}^\la(z_0)}|^2}{(2\rho)^2}\,dx\right)^\frac{1}{2}\left(\fint_{B_{2 \rho}(x_0)}\frac{|u-u_{Q_{2\rho}^\la(z_0)}|^2}{(2\rho)^2}\,dx\right)^\frac{1}{2}\,dt\\
			&\qquad\le c \fint_{I_{2\rho}^\la(t_0)}\left(\fint_{B_{2\rho}(x_0)}\left(\frac{|u-u_{Q_{2\rho}^\la(z_0)}|^{\theta  p}}{(2\rho)^{\theta  p}}+|\na u|^{\theta  p}\right)\,dx\right)^\frac{1}{\theta p}\left(S(u,Q_{2\rho}^\la(z_0))\right)^\frac{1}{2}\,dt\\
			&\qquad\le c\la \left(\fiint_{Q_{2 \rho}^\la(z_0)}\left(\frac{|u-u_{Q_{2\rho}^\la(z_0)}|^{\theta  p}}{(2\rho)^{\theta  p}}+|\na u|^{\theta  p}\right)\,dz\right)^\frac{1}{\theta p},
		\end{split}
	\end{align*}
 where $c=c(\data)$.
Lemma~\ref{sec4:lem:2} implies that
\begin{align*}
	&\la^{p-2}\fiint_{Q_{2\rho}^\la(z_0)}\frac{|u-u_{Q_{2\rho}^\la(z_0)}|^2}{(2\rho)^2}\,dz
	\le  c\la^{p-\beta}\left(\fiint_{Q_{2\rho}^\la(z_0)}H(z,|\na u|)^{\theta}\,dz\right)^{\frac{\beta}{\theta p}}\\
	&\qquad\qquad\qquad+c\la^{p-\beta}\left(\fiint_{Q_{2\rho}^\la(z_0)}H(z,|F|)\,dz\right)^{\frac{\beta}{p}}.
\end{align*}
By combining the estimates above and applying \eqref{sec5:5} and Young's inequality, we obtain
\begin{align*}
	\begin{split}
		&\fiint_{Q_{\rho}^\la(z_0)}H(z,|\na u|)\,dz\\
		&\qquad\le  c \la^{p-\beta}\left(\fiint_{Q_{2\rho}^\la(z_0)}H(z,|\na u|)^{\theta}\,dz\right)^{\frac{\beta}{\theta p}}
		+c\la^{p-\beta}\left(\fiint_{Q_{2\rho}^\la(z_0)}H(z,|F|)\,dz\right)^{\frac{\beta}{p}}\\
  &\qquad\le  \frac{1}{2}\la^p+c\left(\fiint_{Q_{2\rho}^\la(z_0)}H(z,|\na u|)^{\theta}\,dz\right)^{\frac{1}{\theta }}
		+c\fiint_{Q_{2\rho}^\la(z_0)}H(z,|F|)\,dz.
	\end{split}
\end{align*}
The third condition in \eqref{sec5:1} implies that
\begin{align*}
		\fiint_{Q_{\rho}^\la(z_0)}H(z,|\na u|)\,dz
		\le  c\left(\fiint_{Q_{2\rho}^\la(z_0)}H(z,|\na u|)^{\theta}\,dz\right)^{\frac{1}{\theta }}
		+c\fiint_{Q_{2\rho}^\la(z_0)}H(z,|F|)\,dz.
\end{align*}
This completes the proof.
\end{proof}

The following lemma will be used in the next section. 
\begin{lemma}\label{sec5:lem:4}
	Let $u$ be a weak solution to \eqref{sec1:1}. Then there exist constants $c=c(\data)$ and $\theta_0=\theta_0(n,p,q)\in(0,1)$ such that for any $\theta\in(\theta_0,1)$,
	\begin{align*}
		\begin{split}
			&\iint_{Q_{2\cv\rho}^\la(z_0)}H(z,|\na u|)\,dz
			\le c\La^{1-\theta}\iint_{Q_{2\rho}^\la(z_0)\cap \Psi(c^{-1}\La)}H(z,|\na u|)^\theta\,dz\\
			&\qquad\qquad\qquad +c\iint_{Q_{2\rho}^\la(z_0)\cap \Phi(c^{-1}\La)}H(z,|F|)\,dz,
		\end{split}
	\end{align*}
 whenever $Q_{2\cv\rho}^{\la}(z_0)\subset\Omega_T$ satisfies \eqref{sec5:1}.
 Here $\Psi(\La)$ and $\Phi(\La)$ are defined in \eqref{sec2:2} and \eqref{sec2:3}.
\end{lemma}
\begin{proof}
	The second condition in \eqref{sec5:1} implies that 
	\begin{align*}
			\left(\fiint_{Q_{2\rho}^\la(z_0)}H(z,|\na u|)^\theta \,dz\right)^\frac{1}{\theta }
   \le \la^{p(1-\theta)}\fiint_{Q_{2\rho}^\la(z_0)}H(z,|\na u|)^\theta \,dz.
	\end{align*}
     By representing $Q_{2\rho}^\la(z_0)$ as a union of $Q_{2\rho}^\la(z_0)\cap \Psi((4c)^{-1/\theta}\la^p)$ and  $Q_{2\rho}^\la(z_0)\setminus \Psi((4c)^{-1/\theta}\la^p)$ , we have
     \begin{align*}
     		&\left(\fiint_{Q_{2\rho}^\la(z_0)}H(z,|\na u|)^\theta \,dz\right)^\frac{1}{\theta }
     		\le \frac{1}{4c}\la^p +\frac{\la^{p(1-\theta)}}{|Q_{2\rho}^\la|}\iint_{Q_{2\rho}^\la(z_0)\cap \Psi((4c)^{-1/\theta}\la^p)}H(z,|\na u|)^{\theta }\,dz,
     \end{align*}
     for any $c > 0$. 
     A similar argument gives
     \begin{align*}
         \fiint_{Q_{2\rho}^\la(z_0)}H(z,|F|)\,dz\le \frac{1}{4c}\la^p+\iint_{Q_{2\rho}^\la(z_0)\cap \Phi((4c)^{-1}\la^p)}H(z,|F|)\,dz.
     \end{align*}     
    It follows from Lemma~\ref{sec5:lem:3} that
	\begin{align*}
		\begin{split}
			&\fiint_{Q_{\rho}^\la(z_0)}(H(z,|\na u|)+H(z,|F|))\,dz
			\le \frac{1}{2}\la^p+\frac{c\la^{p(1-\theta)}}{|Q_{2\rho}^\la|}\iint_{Q_{2\rho}^\la(z_0)\cap \Psi((4c)^{-1/\theta}\la^p)}H(z,|\na u|)^{\theta }\,dz\\
			&\qquad\qquad\qquad +\frac{c}{|Q_{2\rho}^\la|}\iint_{Q_{2\rho}^\la(z_0)\cap \Phi((4c)^{-1}\la^p)}H(z,|F|)\,dz.
		\end{split}
	\end{align*}
	By recalling the second and third conditions in \eqref{sec5:1}, we obtain
	\begin{align*}
		\begin{split}
			&\fiint_{Q_{2\cv\rho}^\la(z_0)}(H(z,|\na u|)+H(z,|F|))\,dz
			\le 2c\frac{\la^{p(1-\theta)}}{|Q_{2\rho}^\la|}\iint_{Q_{2\rho}^\la(z_0)\cap \Psi((4c)^{-1/\theta}\la^p)}H(z,|\na u|)^{\theta }\,dz\\
			&\qquad\qquad\qquad +\frac{2c}{|Q_{2\rho}^\la|}\iint_{Q_{2\rho}^\la(z_0)\cap \Phi((4c)^{-1}\la^p)}H(z,|F|)\,dz.
		\end{split}
	\end{align*}
    Thus, we have
    \begin{align}\label{sec5:6}
    	\begin{split}
    		&\iint_{Q_{2\cv\rho}^\la(z_0)}H(z,|\na u|)\,dz
    		\le 2c\la^{p(1-\theta)}\iint_{Q_{2\rho}^\la(z_0)\cap \Psi((4c)^{-1/\theta}\la^p)}H(z,|\na u|)^{\theta }\,dz\\
    		&\qquad\qquad\qquad +2c\iint_{Q_{2\rho}^\la(z_0)\cap \Phi((4c)^{-1}\la^p)}H(z,|F|)\,dz.
    	\end{split}
    \end{align}
    
    We note that
    \begin{align*}
    	\frac{\la^p}{4c}\ge \frac{\la^p}{(4c)^{1/\theta}}\ge \frac{\la^p}{(4c)^{1/\theta_0}}\ge \frac{\la^p+a(z_0)\la^q}{2K(4c)^{1/\theta_0}}=\frac{\La}{2K(4c)^{1/\theta_0}},
    \end{align*}
    where we applied the first condition in \eqref{sec5:1}. The estimate above implies that 
    \begin{align*}
    		\Psi((4c)^{-1/\theta}\la^p)\subset \Psi((2K(4c)^{1/\theta_0})^{-1}\La)
      \quad\text{and}\quad
    		\Phi((4c)^{-1}\la^p)\subset \Phi((2K(4c)^{1/\theta_0})^{-1}\La).
    \end{align*}
    Therefore, by replacing $2K(4c)^{1/\theta_0}$ with $c$, \eqref{sec5:6} can be written as
    \begin{align*}
    	\begin{split}
    		&\iint_{Q_{2\cv\rho}^\la(z_0)}H(z,|\na u|)\,dz
    		\le c\La^{1-\theta}\iint_{Q_{2\rho}^\la(z_0)\cap \Psi(c^{-1}\La)}H(z,|\na u|)^{\theta }\,dz\\
    		&\qquad\qquad\qquad +c\iint_{Q_{2\rho}^\la(z_0)\cap \Phi(c^{-1}\La)}H(z,|F|)\,dz.
    	\end{split}
    \end{align*}
    This completes the proof.
\end{proof}

\subsection{The $(p,q)$-intrinsic case}
In this case we consider estimates in $(p,q)$-intrinsic cylinders as in \eqref{sec4:32} and assume that \eqref{sec5:2} holds.
We remark that constants in the estimates depend only on $n,N,p,q,\nu,L$ since \eqref{sec1:1} reduces to a parabolic $(p,q)$-Laplace system in $G_{2\cv\rho}^\la(z_0)$. 
We denote
	\begin{align*}
		S(u,G_{\rho}^\la(z_0))=\sup_{J_{\rho}^\la(t_0)}\fint_{B_{\rho}(x_0)}\frac{|u-u_{G_{\rho}^\la(z_0)}|^2}{\rho^2}\,dx.
	\end{align*}
 
\begin{lemma}\label{sec5:lem:5}
	Let $u$ be a weak solution to \eqref{sec1:1}. Then there exists a constant $c=c(n,N,p,q,\nu,L)$ such that 
	\begin{align*}
		S(u,G_{2\rho}^\la(z_0))=\sup_{J_{2\rho}^\la(t_0)}\fint_{B_{2\rho}(x_0)}\frac{|u-u_{G_{2\rho}^\la(z_0)}|^2}{(2\rho)^2}\,dx\le c\la^2,
	\end{align*}	
 whenever $G_{2\cv\rho}^{\la}(z_0)\subset\Omega_T$ satisfies \eqref{sec5:2}.
\end{lemma}
\begin{proof}
	Let $2\rho\le \rho_1<\rho_2\le 4\rho$. By Lemma~\ref{sec3:lem:1}, there exists a constant $c=c(n,p,q,\nu,L)$ such that
 \begin{equation}\label{sec5:61}
		\begin{split}
			&\frac{H_{z_0}(\la)}{\la^2}\sup_{J_{\rho_1}^\la(t_0)}\fint_{B_{\rho_1}(x_0)}\frac{|u-u_{G_{\rho_1}^\la(z_0)}|^2}{\rho_1^2}\,dx\\
			&\qquad\le \frac{c\rho_2^q}{(\rho_2-\rho_1)^q}\fiint_{G_{\rho_2}^\la(z_0)}\left(\frac{|u-u_{G_{\rho_2}^\la(z_0)}|^p}{\rho_2^p}+a(z)\frac{|u-u_{G_{\rho_2}^\la(z_0)}|^q}{\rho_2^q}\right)\,dz\\
			&\qquad\qquad+\frac{c \rho_2^2}{(\rho_2-\rho_1)^2}\frac{H_{z_0}(\la)}{\la^2}\fiint_{G_{\rho_2}^\la(z_0)}\frac{|u-u_{G_{\rho_2}^\la(z_0)}|^2}{\rho_2^2}\,dz
			+c\fiint_{G_{\rho_2}^\la(z_0)}H(z,|F|)\,dz.
		\end{split}
  \end{equation}
	
	For the first term on the right-hand side of \eqref{sec5:61}, we apply Lemma~\ref{sec4:lem:5} together with the second and third conditions in \eqref{sec5:2} to obtain
	\begin{align*}
		\begin{split}
			&\fiint_{G_{\rho_2}^\la(z_0)}\left(\frac{|u-u_{G_{\rho_2}^\la(z_0)}|^p}{\rho_2^p}+a(z)\frac{|u-u_{G_{\rho_2}^\la(z_0)}|^q}{\rho_2^q}\right)\,dz\\
			&\qquad\le 2\fiint_{G_{\rho_2}^\la(z_0)}H_{z_0}\left(\frac{|u-u_{G_{\rho_2}^\la(z_0)}|}{\rho_2}\right)\,dz\\
   &\qquad\le c\fiint_{Q_{\rho_2}^\la(z_0)}H_{z_0}(|\na u|+|F|)\,dz\le cH_{z_0}(\la),
		\end{split}
	\end{align*}
    where $c=c(n,N,p,q,L)$.
    
	For the second term on the right-hand side of \eqref{sec5:61}, as in the proof of Lemma~\ref{sec5:lem:1}, we obtain
	\begin{align*}
		\begin{split}
			&\fiint_{G_{\rho_2}^\la(z_0)}\frac{|u-u_{G_{\rho_2}^\la(z_0)}|^2}{\rho_2^2}\,dz\\
			&\qquad\le c\left(\fiint_{G_{\rho_2}^\la(z_0)}\left(\frac{|u-u_{G_{\rho_2}^\la(z_0)}|^p}{\rho_2^p}+|\na u|^p\right)\,dz\right)^\frac{1}{p}\left(S(u,G_{\rho_2}^\la(z_0))\right)^\frac{1}{2},
		\end{split}
	\end{align*}
 where $c=c(n,N,p)$.
	By using Lemma~\ref{sec4:lem:6} and \eqref{sec4:5}, we obtain
	\begin{align*}
		\fiint_{G_{\rho_2}^\la(z_0)}\frac{|u-u_{G_{\rho_2}^\la(z_0)}|^2}{\rho_2^2}\,dz\le c\la S(u,G_{\rho_2}^\la(z_0))^\frac{1}{2},
	\end{align*}
 where $c=c(n,N,p,q,L)$.
	By combining estimates and arguing as in the proof of Lemma~\ref{sec5:lem:1}, we have
	\begin{align*}
		\begin{split}
			S(u,G_{\rho_1}^\la(z_0))
			\le& \frac{1}{2}S(u,G_{\rho_2}^\la(z_0))+c\left(\frac{\rho_2^q}{(\rho_2-\rho_1)^q}+\frac{ \rho_2^4}{(\rho_2-\rho_1)^4}\right)\la^2.
		\end{split}
	\end{align*}
	The conclusion follows by applying Lemma~\ref{sec2:lem:2}.
\end{proof}

\begin{lemma}\label{sec5:lem:7}
	Let $u$ be a weak solution to \eqref{sec1:1}. Then there exists a constant $c=c(n,p,q)$ such that for any $\theta\in(n/(n+2),1)$,
	\begin{align*}
		\begin{split}
			&\fiint_{G_{2\rho}^\la(z_0)}\left(\frac{|u-u_{G_{2\rho}^\la(z_0)}|^p}{(2\rho)^p}+a(z)\frac{|u-u_{G_{2\rho}^\la(z_0)}|^q}{(2\rho)^q}\right)\,dz\\
			&\qquad\le  c\fiint_{G_{2\rho}^\la(z_0)}\left(H^{\theta}_{z_0}\left(\frac{|u-u_{G_{2\rho}^\la(z_0)}|}{2\rho}\right)+H_{z_0}^{\theta}(|\na u|)\right)\,dz\,H_{z_0}^{1-\theta}\left(S(u,G_{2\rho}^\la(z_0))^{\frac{1}{2}}\right),
		\end{split}
	\end{align*}
  whenever $G_{2\cv\rho}^{\la}(z_0)\subset\Omega_T$ satisfies \eqref{sec5:2}.
\end{lemma}

\begin{proof}
	From the second condition in \eqref{sec5:2}, we obtain
	\begin{align*}
		\begin{split}
			&\fiint_{G_{2\rho}^\la(z_0)}\left(\frac{|u-u_{G_{2\rho}^\la(z_0)}|^p}{(2\rho)^p}+a(z)\frac{|u-u_{G_{2\rho}^\la(z_0)}|^q}{(2\rho)^q}\right)\,dz\\
			&\qquad\le2\fiint_{G_{2\rho}^\la(z_0)}\left(\frac{|u-u_{G_{2\rho}^\la(z_0)}|^p}{(2\rho)^p}+a(z_0)\frac{|u-u_{G_{2\rho}^\la(z_0)}|^q}{(2\rho)^q}\right)\,dz.
		\end{split}
	\end{align*}
	By Lemma~\ref{sec2:lem:1}, there exists a constant $c=c(n,p,q)$ such that for any $\theta\in(n/(n+2),1)$, we have
	\begin{align*}
		\begin{split}
			&\fiint_{G_{2\rho}^\la(z_0)}\frac{|u-u_{G_{2\rho}^\la(z_0)}|^p}{(2\rho)^p}\,dz\\
			&\qquad\le
			c\fiint_{G_{2\rho}^\la(z_0)}\left(\frac{|u-u_{G_{2\rho}^\la(z_0)}|^{\theta p}}{(2\rho)^{\theta p}}+|\na u|^{\theta p}\right)\,dz\left(S(u,G_{2\rho}^\la(z_0))^\frac{1}{2}\right)^{(1-\theta)p}
		\end{split}
	\end{align*}
	and
	\begin{align*}
		\begin{split}
			&\fiint_{G_{2\rho}^\la(z_0)}a(z_0)\frac{|u-u_{G_{2\rho}^\la(z_0)}|^q}{(2\rho)^q}\,dz\\		
			&\le
			c\fiint_{G_{2\rho}^\la(z_0)}\left(a(z_0)^{\theta}\frac{|u-u_{G_{2\rho}^\la(z_0)}|^{\theta q}}{(2\rho)^{\theta q}}+a(z_0)^{\theta}|\na u|^{\theta q}\right)\,dz\
			 a(z_0)^{1-\theta}\left(S(u,G_{2\rho}^\la(z_0))^\frac{1}{2}\right)^{(1-\theta)q}.
		\end{split}
	\end{align*}
	Thus we conclude that
	\begin{align*}
		\begin{split}
			&\fiint_{G_{2\rho}^\la(z_0)}\left(\frac{|u-u_{G_{2\rho}^\la(z_0)}|^p}{(2\rho)^p}+a(z)\frac{|u-u_{G_{2\rho}^\la(z_0)}|^q}{(2\rho)^q}\right)\,dz\\
			&\qquad\le  c\fiint_{G_{2\rho}^\la(z_0)}H^{\theta}_{z_0}\left(\frac{|u-u_{G_{2\rho}^\la(z_0)}|}{2\rho}\right)+H_{z_0}^{\theta}(|\na u|)\,dz\\
			&\qquad\qquad \times \left(\left(S(u,G_{2\rho}^\la(z_0))^\frac{1}{2}\right)^{p}+a(z_0)\left(S(u,G_{2\rho}^\la(z_0))^\frac{1}{2}\right)^{q}\right)^{1-\theta }.
		\end{split}
	\end{align*}
	This completes the proof.	
\end{proof}

\begin{lemma}\label{sec5:lem:6}
	Let $u$ be a weak solution to \eqref{sec1:1}. Then there exist constants $c=c(n,N,p,q,\nu,L)$ and $\theta_0=\theta_0(n,p,q)\in(0,1)$ such that for any $\theta\in(\theta_0,1)$,
	\begin{align*}
			\fiint_{G_{\rho}^\la(z_0)}H_{z_0}(|\na u|+|F|)\,dz
			\le c\left(\fiint_{G_{2\rho}^\la(z_0)}H_{z_0}^{\theta}(|\na u|)\,dz\right)^\frac{1}{\theta }+c\fiint_{G_{2\rho}^\la(z_0)} H_{z_0}(|F|)\,dz,
	\end{align*}
 whenever $G_{2\cv\rho}^{\la}(z_0)\subset\Omega_T$ satisfies \eqref{sec5:2}.
	Moreover, we have
	\begin{align*}
		\begin{split}
			&\iint_{G_{2\cv\rho}^\la(z_0)}H(z,|\na u|)\,dz
			\le c\La^{1-\theta}\iint_{G_{2\rho}^\la(z_0)\cap \Psi(c^{-1}\La)}H(z,|\na u|)^\theta\,dz\\
			&\qquad\qquad\qquad+c\iint_{G_{2\rho}^\la(z_0)\cap \Phi(c^{-1}\La)}H(z,|F|)\,dz,
		\end{split}
	\end{align*}
	where $\Psi(\La)$ and $\Phi(\La)$ are as in \eqref{sec2:2} and \eqref{sec2:3}.
	
\end{lemma}

\begin{proof}
	Once the first estimate in the statement holds, then the second estimate follows as in the proof of Lemma~\ref{sec5:lem:4}.
	
	To prove the first estimate in the statement, we apply Lemma~\ref{sec3:lem:1} to obtain
    \begin{equation}\label{sec5:62}
		\begin{split}
			&\fiint_{G_{\rho}^\la(z_0)}H(z,|\na u|)\,dz
			\le c\fiint_{G_{2\rho}^\la(z_0)}\left(\frac{|u-u_{G_{2\rho}^\la(z_0)}|^p}{(2\rho)^p}+a(z_0)\frac{|u-u_{G_{2\rho}^\la(z_0)}|^q}{(2\rho)^q}\right)\,dz\\
			&\qquad\qquad+c\frac{H_{z_0}(\la)}{\la^2}\fiint_{G_{2\rho}^\la(z_0)}\frac{|u-u_{G_{2\rho}^\la(z_0)}|^2}{(2\rho)^2}\,dz
			+c\fiint_{G_{2\rho}^\la(z_0)}H(z,|F|)\,dz.
		\end{split}
	\end{equation}
	
	Using Lemma~\ref{sec5:lem:7}, Lemma~\ref{sec4:lem:5} and Lemma~\ref{sec5:lem:5} for the first term on the right-hand side of \eqref{sec5:62}, we obtain
	\begin{align*}
		\begin{split}
			&\fiint_{G_{2\rho}^\la(z_0)}\left(\frac{|u-u_{G_{2\rho}^\la(z_0)}|^p}{(2\rho)^p}+a(z)\frac{|u-u_{G_{2\rho}^\la(z_0)}|^q}{(2\rho)^q}\right)\,dz\\
			&\qquad\le c\fiint_{G_{2\rho}^\la(z_0)}H^\theta_{z_0}(|\na u|)\,dz\,H_{z_0}^{1-\theta}(\la)
			+ c\left( \fiint_{G_{2\rho}^\la(z_0)} H_{z_0}(|F|)\,dz \right)^\theta H_{z_0}^{1-\theta}(\la).
		\end{split}
	\end{align*}
    As in the proof of Lemma~\ref{sec5:lem:3}, we obtain
	\begin{align*}
			\fiint_{G_{2\rho}^\la(z_0)}\frac{|u-u_{G_{2\rho}^\la(z_0)}|^2}{(2\rho)^2}\,dz
			\le c\la\left(\fiint_{G_{2\rho}^\la(z_0)}\left(\frac{|u-u_{G_{2\rho}^\la(z_0)}|^{\theta p}}{(2\rho)^p}+|\na u|^{\theta p}\right)\,dz\right)^\frac{1}{\theta p}
	\end{align*}
	and from Lemma~\ref{sec4:lem:6} we conclude that
	\begin{align*}
		\fiint_{G_{2\rho}^\la(z_0)}\frac{|u-u_{G_{2\rho}^\la(z_0)}|^{\theta p}}{(2\rho)^p}\,dz\le c\fiint_{G_{2\rho}^\la(z_0)}|\na u|^{\theta p}\,dz+c\left(\fiint_{G_{2\rho}^{\la}(z_0)}|F|^p\,dz\right)^{\theta }.
	\end{align*}
 
    For the second term on the right-hand side of \eqref{sec5:62}, we have
    \begin{align*}
    	\begin{split}
    		&\frac{H_{z_0}(\la)}{\la^2} \fiint_{G_{2\rho}^\la(z_0)}\frac{|u-u_{G_{2\rho}^\la(z_0)}|^2}{(2\rho)^2}\,dz
    		\le \frac{H_{z_0}(\la)}{\la}\left(\fiint_{G_{2\rho}^\la(z_0)}|\na u|^{\theta p}\,dz\right)^\frac{1}{\theta p}\\
    		&\qquad\qquad\qquad+c\frac{H_{z_0}(\la)}{\la}\left(\fiint_{G_{2\rho}^{\la}(z_0)}|F|^p\,dz\right)^{\frac{1}{p}},
    	\end{split}
    \end{align*}
    where
    \begin{align*}
    	\begin{split}
    		&\frac{H_{z_0}(\la)}{\la}\left(\fiint_{G_{2\rho}^\la(z_0)}|\na u|^{\theta p}\,dz\right)^\frac{1}{\theta p}
    		\le \la^{p-1}\left(\fiint_{G_{2\rho}^\la(z_0)}|\na u|^{\theta p}\,dz\right)^\frac{1}{\theta p}\\
    		&\qquad\qquad\qquad+\left(a(z_0)\la^q\right)^\frac{q-1}{q}\left(\fiint_{G_{2\rho}^{\la}(z_0)}a(z_0)^\theta|\na u|^{\theta q}\,dz\right)^\frac{1}{\theta q}.
    	\end{split}
    \end{align*}
    A similar argument for $|F|$ gives
    \begin{align*}
    	\begin{split}
    		&\frac{H_{z_0}(\la)}{\la^2}\fiint_{G_{2\rho}^\la(z_0)}\frac{|u-u_{G_{2\rho}^\la(z_0)}|^2}{(2\rho)^2}\,dz\\
    		&\qquad\le \la^{p-1}\left(\fiint_{G_{2\rho}^\la(z_0)}|\na u|^{\theta p}\,dz\right)^\frac{1}{\theta p}
    		+\left(a(z_0)\la^q\right)^\frac{q-1}{q}\left(\fiint_{G_{2\rho}^{\la}(z_0)}a(z_0)^\theta|\na u|^{\theta q}\,dz\right)^\frac{1}{\theta q}\\
    		 &\qquad\qquad+\la^{p-1}\left(\fiint_{G_{2\rho}^\la(z_0)}|F|^{ p}\,dz\right)^\frac{1}{p}
    		+\left(a(z_0)\la^q\right)^\frac{q-1}{q}\left(\fiint_{G_{2\rho}^{\la}(z_0)}a(z_0)|F|^{q}\,dz\right)^\frac{1}{ q}.
    	\end{split}    	
    \end{align*}

	By collecting all the estimates above and applying Young's inequality together with the second condition in \eqref{sec5:2}, we obtain
	\begin{align*}
		\begin{split}
			&\fiint_{G_{\rho}^\la(z_0)}H(z,|\na u|)\,dz
			\le \frac{1}{2}H_{z_0}(\la)
   +c\left(\fiint_{G_{2\rho}^\la(z_0)}H(z,|\na u|)^{\theta}\,dz\right)^{\frac{1}{\theta}}\\
			&\qquad\qquad\qquad+c\fiint_{G_{2\rho}^\la(z_0)}H(z,|F|)\,dz.
		\end{split}
	\end{align*}
	We use the fourth condition in \eqref{sec5:2} to absorb the first term on the right-hand side. This completes the proof.	
\end{proof}

\section{The proof of Theorem~\ref{main_theorem}}
In this section, we will complete the proof of Theorem~\ref{main_theorem}. We divide the section into three subsections. 
In the first subsection, we construct intrinsic cylinders which are either $p$-intrinsic or $(p,q)$-intrinsic, see \eqref{sec5:1} and \eqref{sec5:2}. In the second subsection, we prove a Vitali type covering property for the system of intrinsic cylinders constructed in the first subsection. Note that the collection consists of two different types of intrinsic cylinders depending on the center point of the cylinders. Finally, in the last subsection, we complete the proof of gradient estimate by applying Fubini's theorem together with Lemma~\ref{sec2:lem:2}.

\subsection{Stopping time argument}\label{stopping time argument}
Let
\begin{equation}\label{sec6:01}
\begin{split}
		\la_0^2&=\fiint_{Q_{2r}(z_0)}\left(H(z,|\na u|)+H(z,|F|)\right)\,dz+1,\\
		\La_0&=\la_0^p+\sup_{z\in Q_{2r}(z_0)}a(z)\la_0^q
		\quad\text{and}\quad
		\gamma=\frac{\alpha p}{n+2},
  \end{split}
\end{equation}
where $Q_{2r}(z_0)=B_{2r}(x_0)\times (t_0-(2r)^2,t_0+(2r)^2)$.
Moreover, let
\begin{align}\label{sec6:1}
	K=1+40[a]_{\alpha} M_1^\frac{\alpha}{n+2}\text{and}\quad
 \cv=10K.
\end{align}
With $\Psi(\La)$ and $\Phi(\La)$ as in \eqref{sec2:2}-\eqref{sec2:3} and $\rho\in[r,2r]$, we denote
\[
\Psi(\La,\rho)=\Psi(\La)\cap Q_{\rho}(z_0)=\{ z\in Q_{\rho}(z_0): H(z,|\na u(z)|)>\La\} 
\]
and
\[
\Phi(\La,\rho)=\Phi(\La)\cap Q_{\rho}(z_0)
=\{ z\in Q_{\rho}(z_0): H(z,|F(z)|)>\La\}.
\]

Next we apply a stopping time argument. Let $r\le r_1<r_2\le 2r$ and 
\begin{align}\label{sec6:4}
	\La>\left(\frac{4\cv r}{r_2-r_1}\right)^\frac{q(n+2)}{2}\La_0,
\end{align}
where $\cv$ is as in \eqref{sec6:1}.
For every $\mz\in \Psi(\La,r_1)$, let $\laz>0$ be such that
\begin{align}\label{sec6:5}
	\La=\laz^p+a(\mz)\laz^q=H_{\mz}(\laz),
\end{align}
where $H_{\mz}$ is as in \eqref{sec2:1}.
We claim that 
\begin{align*}
	\laz>\left(\frac{4\cv r}{r_2-r_1}\right)^\frac{n+2}{2}\la_0.
\end{align*}
For a contradiction, assume that the inequality above does not hold. Then
\begin{align*}
	\La< \left(\frac{4\cv r}{r_2-r_1}\right)^{\frac{q(n+2)}{2}}\left(\la_0^p+a(\mz)\la_0^q\right)\le \left(\frac{ 4\cv r}{r_2-r_1}\right)^{\frac{q(n+2)}{2}}\La_0,
\end{align*}
which is a contradiction with \eqref{sec6:4}.
Therefore, for $s\in [(r_2-r_1)/(2\cv),r_2-r_1)$, we have 
\begin{align} \label{sec6:5_1}
	\begin{split}
		&\fiint_{Q_s^{\laz}(\mz)}(H(z,|\na u|)+H(z,|F|))\,dz\\
		&\qquad\le\laz^{p-2}\left(\frac{2r}{s}\right)^{n+2}\fiint_{Q_{2r}(z_0)}(H(z,|\na u|)+H(z,|F|))\,dz\\
		&\qquad\le\left(\frac{4\cv r}{r_2-r_1}\right)^{n+2}\laz^{p-2}\la_0^2< \laz^p.
	\end{split}
\end{align}
Since $\mz\in \Psi(\La,r_1)$ and \eqref{sec6:5} holds, it follows that $\mz\in \Psi(\laz^p,r_1)$. By the Lebesgue differentiation theorem there exists $\rho_{\mz}\in(0,(r_2-r_1)/(2\cv))$ such that
\begin{align}\label{sec6:6}
	\fiint_{Q_{\rho_{\mz}}^{\laz}(\mz)}(H(z,|\na u|)+H(z,|F|))\,dz= \laz^p
\end{align}
and 
\begin{align}\label{sec6:7}
	\fiint_{Q_s^{\laz}(\mz)}(H(z,|\na u|)+H(z,|F|))\,dz<\laz^p
\end{align}
for every $s\in(\rho_{\mz},r_2-r_1)$.
Observe that~\eqref{sec6:6} and~\eqref{sec6:5_1}, imply that
\begin{equation} \label{sec6:72}
\laz \leq \left( \frac{2r}{\rho_{\mz}} \right)^\frac{n+2}{2} \la_0.
\end{equation}

For $K>1$ as in \eqref{sec6:1}, either 
\begin{align}\label{sec6:8}
	K\laz^p\ge a(\mz)\laz^q
	\quad\text{or}\quad 
	K\laz^p\le a(\mz)\laz^q.
\end{align}
In addition, either
\begin{align}\label{sec6:9}
	a(\mz)\ge 2[a]_{\alpha}(10\rho_{\mz})^\alpha
	\quad\text{or}\quad a(\mz)\le 2[a]_{\alpha}(10\rho_{\mz})^\alpha.
\end{align}
We consider three cases:
\begin{enumerate}[label=(\arabic*)]
\item\label{case1} $K\laz^p\ge a(\mz)\laz^q$, that is, the first condition in \eqref{sec6:8} holds,
\item\label{case2} $K\laz^p\le a(\mz)\laz^q$ and $a(\mz)\ge 2[a]_{\alpha}(10\rho_{\mz})^\alpha$, that is, the second condition in \eqref{sec6:8} and the first condition in \eqref{sec6:9} and
\item\label{case3} $K\laz^p\le a(\mz)\laz^q$ and $a(\mz)\le 2[a]_{\alpha}(10\rho_{\mz})^\alpha$, that is, the second condition in \eqref{sec6:8} and the second condition in \eqref{sec6:9} hold.
\end{enumerate}

First we note that \ref{case1}, together with \eqref{sec6:6}-\eqref{sec6:7}, imply \eqref{sec5:1} for $p$-intrinsic cylinders by replacing the center point and radius with $\mz$ and $\rho_{\mz}$. 
Next we show that \ref{case2} implies \eqref{sec5:2} for $(p,q)$-intrinsic cylinders. From the second condition in \eqref{sec6:8} we obtain $a(\mz)>0$ and $G_{s}^{\la_{\mz}}(w)\subsetneq Q_s^{\la_{\mz}}(w)$. By \eqref{sec6:6}-\eqref{sec6:7}, we obtain
\begin{align*}
		&\fiint_{G_{s}^{\laz}(\mz)}(H(z,|\na u|)+H(z,|F|))\,dz\\
		&\qquad\leq\frac{H_{\mz}(\laz)}{\laz^p}\fiint_{Q_{s}^{\laz}(\mz)}(H(z,|\na u|)+H(z,|F|))\,dz< H_{\mz}(\laz)
\end{align*}
for every $s\in(\rho_{\mz},r_2-r_1)$.
Recall that $\mz\in \Psi(\La,r_1)$ and $\La=H_{\mz}(\laz)$. We find $\varsigma_{\mz}\in(0,\rho_{\mz}]$ such that
\begin{align}\label{sec6:10}
	\fiint_{G_{\varsigma_{\mz}}^{\laz}(\mz)}H(z,|\na u|)+H(z,|F|)\,dz= H_{\mz}(\laz)
\end{align}
and
\begin{align*}
	\fiint_{G_s^{\laz}(\mz)}(H(z,|\na u|)+H(z,|F|))\,dz<H_{\mz}(\laz)
\end{align*}
for every $s\in(\varsigma_{\mz},r_2-r_1)$.
Moreover, it follows from the first condition in \eqref{sec6:9} that
\begin{align*}
	2[a]_{\alpha}(10\rho_{\mz})^\alpha\le a(\mz)\le \inf_{Q_{10\rho_{\mz}}(\mz)}a(z)+[a]_{\alpha}(10\rho_{\mz})^\alpha
\end{align*}
and
\begin{align*}
	\sup_{Q_{10\rho_{\mz}}(\mz)}a(z)\le \inf_{Q_{10\rho_{\mz}}(\mz)}a(z)+[a]_{\alpha}(10\rho_{\mz})^\alpha\le 2\inf_{Q_{10\rho_{\mz}}(\mz)}a(z).
\end{align*}
Therefore, 
\begin{align}\label{sec6:12}
	\frac{a(\mz)}{2}\le a(z)\le 2a(\mz)
	\text{ for every }
	z\in Q_{10\rho_{\mz}}(\mz).
\end{align}
Hence, \ref{case2} implies \eqref{sec6:10}-\eqref{sec6:12}. This shows that \eqref{sec5:2} is satisfied by replacing the center point and radius with $\mz$ and $\varsigma_{\mz}$.

Finally, we prove that \ref{case3} never occurs due to \eqref{sec6:1}. From the second condition in \eqref{sec6:8} and the second condition in \eqref{sec6:9}, we have
\begin{align*}
	K\laz^p=a(\mz)\frac{K\laz^p}{a(\mz)}\le 20[a]_{\alpha}\rho_{\mz}^{\alpha}\laz^q.
\end{align*}
By applying \eqref{sec6:6} and recalling that $\gamma=\tfrac{\alpha p}{n+2}$, we obtain
\begin{align*}
	K\laz^p\le 20[a]_{\alpha}\rho_{\mz}^\alpha\left(\fiint_{Q^{\laz}_{\rho_{\mz}} (\mz)}(H(z,|\na u|)+H(z,|F|))\,dz\right)^\frac{\gamma}{p}\laz^{q-\gamma}.
\end{align*}
Observe that
\begin{align*}
	\begin{split}
		&\rho_{\mz}^\alpha \left(\fiint_{Q_{\rho_{\mz}}^{\laz}(\mz)}(H(z,|\na u|)+H(z,|F|))\,dz\right)^\frac{\gamma}{p}\laz^{q-\gamma}\\
		&\qquad\le \rho_{\mz}^{\alpha-\frac{\gamma(n+2)}{p}}M_1^{\frac{\gamma}{p}}\laz^{q-\gamma+(p-2)\frac{\gamma}{p}}=M_1^\frac{\alpha}{n+2}\laz^{q-\gamma+(p-2)\frac{\gamma}{p}}
	\end{split}
\end{align*}
and
\begin{align*}
	q-\gamma+(p-2)\frac{\gamma}{p}=q-\frac{2\gamma}{p}=q-\frac{2\alpha }{n+2}\le p.
\end{align*}
It follows from \eqref{sec6:1} that
\begin{align*}
	\laz^p\le \frac{1}{2}\laz^{q-\gamma+(p-2)\frac{\gamma}{p}}\le \frac{1}{2}\laz^p.
\end{align*}
Therefore, the second condition in \eqref{sec6:8} and the second condition in \eqref{sec6:9} cannot occur together.

\subsection{Vitali type covering argument}
In section~\ref{section5} we considered reverse H\"older inequality, and in subsection~\ref{stopping time argument} we discussed a stopping time argument, for $p$-intrinsic and $(p,q)$-intrinsic cylinders. For each $\mz\in \Psi(\La,r_1)$, we consider
\[
\mathcal{Q}(\mz)=Q_{2\rho_{\mz}}^{\laz}(\mz)\text{ in \ref{case1}}
\quad\text{and}\quad
\mathcal{Q}(\mz)=G_{2\varsigma_{\mz}}^{\laz}(\mz)\text{ in \ref{case2}}.
\]

We prove a Vitali type covering lemma for this collection of intrinsic cylinders.
We denote
\begin{align*}
	\mathcal{F}=\left\{\mathcal{Q}(\mz): \mz\in \Psi(\La,r_1)\right\}
	\quad\text{and}\quad 
	l_{\mz}=
	\begin{cases}
		2\rho_{\mz}&\text{in\,\ref{case1}},\\
		2\varsigma_{\mz}&\text{in\,\ref{case2}}.
	\end{cases}
\end{align*}
Recall that $l_{\mz}\in (0,R)$ for every $\mz\in \Psi(\La,r_1)$, where $R=(r_2-r_1)/\cv$ and $\cv$ is as in \eqref{sec6:1}. Let
\begin{align*}
	\mathcal{F}_j=\left\{\mathcal{Q}(\mz)\in \mathcal{F}: \frac{R}{2^j}<l_{\mz}<\frac{R}{2^{j-1}} \right\},\quad j\in\mathbb N.
\end{align*}
We construct subcollections $\mathcal{G}_j\subset \mathcal{F}_j$, $j\in\mathbb N$, recursively as follows. Let $\mathcal{G}_1$ be a maximal disjoint collection of cylinders in $\mathcal{F}_1$.
By~\eqref{sec6:72} we observe that the measure of each cylinder in $\mathcal{G}_1$ is bounded from below, which implies that the collection is finite. Suppose that we have selected $\mathcal{G}_1,...,\mathcal{G}_{k-1}$ with $k\ge2$, and let
\begin{align*}
	\mathcal{G}_k=\Bigl\{ \mathcal{Q}(\mz)\in \mathcal{F}_k: \mathcal{Q}(\mz)\cap \mathcal{Q}(\mw)=\emptyset\text{ for every }\mathcal{Q}(\mw)\in \bigcup_{j=1}^{k-1}\mathcal{G}_j\Bigr\}    
\end{align*}
be a maximal collection of pairwise disjoint cylinders.
It follows that
\begin{align}\label{sec6:7_2}
	\mathcal{G}=\bigcup_{j=1}^\infty\mathcal{G}_j,
\end{align}
is a countable subcollection of pairwise disjoint cylinders in $\mathcal{F}$. We claim that for each $\mathcal{Q}(\mz)\in \mathcal{F}$, there exists $\mathcal{Q}(\mw)\in \mathcal{G}$ such that
\begin{align}\label{sec6:1_2}
	\mathcal{Q}(\mz)\cap \mathcal{Q}(\mw)\ne\emptyset
	\quad\text{and}\quad 
	\mathcal{Q}(\mz)\subset \cv\mathcal{Q}(\mw),
\end{align}
where 
\[
\cv\mathcal{Q}(\mw)=Q_{2\cv \rho_{\mw}}^{\la_{\mw}}(\mw)\text{ in \ref{case1}}
\quad\text{and}\quad 
\cv\mathcal{Q}(\mw)=G_{2\cv \varsigma_{\mw}}^{\la_{\mw}}(\mw)\text{ in \ref{case2}}.
\]
For every $\mathcal{Q}(\mz)\in \mathcal{F}$, there exists $j \in \mathbb{N}$ such that $\mathcal{Q}(\mz)\in \mathcal{F}_j$. 
By the construction of $\mathcal{G}_j$, there exists a cylinder 
$\mathcal{Q}(\mw)\in\cup_{i=1}^j \mathcal{G}_i$
for which the first condition in \eqref{sec6:1_2} holds true. 
Moreover, since $l_{\mz}\le \tfrac{R}{2^{j-1}}$
and $l_{\mw} \geq \tfrac{R}{2^j}$, we have 
\begin{align}\label{sec6:3_2}
	l_{\mz}\le 2l_{\mw}.
\end{align}

In the remaining of this subsection, we will prove the second claim in \eqref{sec6:1_2}. 
We note that if $\la=\la_{\mz}=\la_{\mw}$ and either 
\[
\mathcal{Q}(\mz)=Q_{l_{\mz}}^{\la}(\mz)
\quad\text{and}\quad 
\mathcal{Q}(\mw)=Q_{l_{\mw}}^{\la}(\mw)
\]
or 
\[
\mathcal{Q}(\mz)=G_{l_{\mz}}^{\la}(\mz)
\quad\text{and}\quad 
\mathcal{Q}(\mw)=G_{l_{\mw}}^{\la}(\mw),
\] 
then the second claim in \eqref{sec6:1_2} holds true if $\cv \geq5$. Indeed, once the scaling factor in the time interval of two intrinsic cylinders is the same, these cylinders are in the standard parabolic metric space. Thus the standard proof of Vitali's covering lemma can be applied in these cases.

Regardless of \ref{case1} and \ref{case2}, for $i\in\{v,w\}$, there exist $2\rho_i\ge l_i>0$ and $\la_i>0$ such that
\begin{align}\label{sec6:14}
	\La=\la_i^p+a(z_i)\la_i^q
\end{align}
and
\begin{align}\label{sec6:15}
	\fiint_{Q_{\rho_i}^{\la_i}(z_i)}(H(z,|\na u|)+H(z,|F|))\,dz=\la_i^p.
\end{align}
We show that the second claim in~\eqref{sec6:1_2} holds in all four possible cases that may occur:
\begin{enumerate}[label=(\roman*)]
\item\label{i} $\mathcal{Q}(v)=Q_{l_v}^{\la_v}(v)$ and $\mathcal{Q}(w)=Q_{l_w}^{\la_w}(w)$,
\item\label{ii} $\mathcal{Q}(v)=G_{l_v}^{\la_v}(v)$ and $ \mathcal{Q}(w)=G_{l_w}^{\la_w}(w)$,
\item\label{iii} $\mathcal{Q}(v)=G_{l_v}^{\la_v}(v)$ and $ \mathcal{Q}(w)=Q^{\la_w}_{l_w}(w)$ and
\item\label{iv} $\mathcal{Q}(v)=Q_{l_v}^{\la_v}(v)$ and $\mathcal{Q}(w)=G_{l_w}^{\la_w}(w)$.
\end{enumerate} 
Observe that in any of these cases, the first condition in~\eqref{sec6:1_2} implies that $Q_{l_{w}}(w) \cap Q_{l_{v}}(v) \neq \emptyset$ and 
\begin{equation} \label{sec6:152}
Q_{l_{w}}(w) \subset Q_{5l_{v}}(v)\subset Q_{10\rho_v}(v).
\end{equation}
This will already imply that the second claim in~\eqref{sec6:1_2} holds for the spatial part of the set by enlarging the radius by factor 5. In the rest of this subsection, we show the inclusion of the time intervals when enlarging the radius with factor $\cv$ by considering each case separately. 

First we collect a few facts that will be applied in the argument. By~\eqref{sec6:152} we have
\begin{equation} \label{sec6:153}
| a(w) - a(v) | \leq [a]_\alpha (10  \rho_v)^\alpha.
\end{equation}
On the other hand, from~\eqref{sec6:15}, we may deduce that 
\begin{align} \label{sec6:154}
\rho_v^{n+2} = \frac{1}{ 2 \left| B_1 \right| \lambda_v^2 } \iint_{Q_{\rho_{v}}^{\la_v}(v)}(H(z,|\na u|)+H(z,|F|))\,dz\le \frac{M_1}{\lambda_v^2},
\end{align}

If $\la_w \le \la_v$, we claim that 
\begin{align}\label{sec6:18}
	\la_v\le\left(2\left(1+10[a]_{\alpha } M_1^\frac{\alpha}{n+2}\right)\right)^\frac{1}{p}\la_w.
\end{align}
For a contraction, assume that \eqref{sec6:18} does not hold. It follows from \eqref{sec6:14} and \eqref{sec6:153} that
\begin{align}\label{sec6:188}
	\begin{split}
		\La=\la_w^p+a(w)\la_w^q\le \la_w^p+a(v)\la_w^q+[a]_{\alpha}(10\rho_v)^{\alpha}\la_w^q.
	\end{split}
\end{align}
By~\eqref{sec6:154} we obtain
\begin{align*}
\rho_v^\alpha \lambda_w^q \leq  M_1^\frac{\alpha}{n+2} \lambda_v^{-\frac{2\alpha}{n+2} } \lambda_w^q < M_1^\frac{\alpha}{n+2} \lambda_w^{q-\frac{2\alpha}{n+2}} \leq M_1^\frac{\alpha}{n+2} \lambda_w^p,
\end{align*}
since $\lambda_w \le \lambda_v$ and $q \leq p + 2\alpha/(n+2)$. 
Substituting the negation of \eqref{sec6:18} and the above display into the right-hand side of \eqref{sec6:188} leads to a contradiction since
\begin{align*}
	\La < \frac{1}{2}\left(\la_v^p+a(v)\la_v^q\right)=\frac{1}{2}\La.
\end{align*}
On the other hand, if $\la_v\le \la_w$, we claim that
\begin{align*}
    \la_w\le\left(2\left(1+10[a]_\alpha M_1^\frac{\alpha}{n+2}\right)\right)^\frac{1}{p}\la_v.
\end{align*}
Otherwise, it follows from \eqref{sec6:154} that
\begin{align*}
\begin{split}
    \La&=\la_v^p+a(v)\la_v^q\le \la_v^p+a(w)\la_v^q+[a]_{\alpha}(10\rho_v)^\alpha\la_v^q\\
    &\le \la_v^p+a(w)\la_v^q+10[a]_\alpha M_1^\frac{\alpha}{n+2}\la_v^{q-\frac{2\alpha}{n+2}}\\
    &\le \left(1+10[a]_\alpha M_1^\frac{\alpha}{n+2}\right)\la_v^p+a(w)\la_v^q< \frac{1}{2}(\la_w^p+a(w)\la_w^q)=\frac{1}{2}\La.
\end{split}
\end{align*}
In any case we have
\begin{equation} \label{sec6:181}
(2K)^{-\frac{1}{p}}\la_w\le \la_v \leq (2K)^\frac{1}{p} \la_w.
\end{equation}
Let $v=(x_v,t_v)$ and $w=(x_w,t_w)$ for $x_v,x_w\in \RR^n$ and $t_v,t_w\in \RR$.

\ref{i}: $\mathcal{Q}(v)=Q_{l_v}^{\la_v}(v)$ and $\mathcal{Q}(w)=Q_{l_w}^{\la_w}(w)$.
\,For any $\tau \in I_{l_w}^{\la_w}(t_w)$, we apply \eqref{sec6:3_2},~\eqref{sec6:181} and $(p-2)/p\le 1$ to have
\begin{align*}
\begin{split}
    |\tau - t_v| &\leq |\tau - t_w| + |t_w - t_v| \leq 2\la_w^{2-p} l_w^2 + \la_v^{2-p} l_v^2\\
    &\leq \left(16 K +1 \right) \la_v^{2-p} l_v^{2}\le 100K^2\la_v^{2-p}l_v^2= \la_v^{2-p}(\cv l_v)^2,
\end{split}
\end{align*}
which implies $I_{l_w}^{\la_w}(t_w) \subset \cv I_{ l_\mw}^{\la_\mw}(t_v)$. Thus, we have
$Q_{l_w}^{\la_w}(w) \subset \cv Q_{ l_v}^{\la_v}(v)$.

\ref{ii}: $\mathcal{Q}(v)=G_{l_v}^{\la_v}(v)$ and $ \mathcal{Q}(w)=G_{l_w}^{\la_w}(w)$.
For any $\tau \in J_{l_w}^{\la_w}(t_w)$, we have
\begin{align*}
    |\tau - t_v| \leq |\tau - t_w| + |t_w - t_v|\leq 2\frac{\la_w^2}{H_{w}(\la_w)} l_w^2 + \frac{\la_v^2}{H_{v}(\la_v)} l_v^2.
\end{align*}
By~\eqref{sec6:14},~\eqref{sec6:181} and $2/p\le1$ we have
\begin{align*}
	\frac{\la_w^2}{H_{w}(\la_w)}=\frac{\la_w^2}{\La}\le 2K\frac{\la_v^2}{\La} = 2K\frac{\la_v^2}{H_{v}(\la_v)}.
\end{align*}
Therefore applying~\eqref{sec6:3_2}, we obtain
\begin{align*}
    |\tau - t_v| \leq (16K+1)\frac{\la_v^2}{H_{v}(\la_v)} l_v^2\le 100K^2\frac{\la_v^2}{H_{v}(\la_v)}l_v^2= \frac{\la_v^2}{H_{v}(\la_v)^2}(\cv l_v)^2.
\end{align*}
This implies that $J_{l_w}^{\la_w}(t_w) \subset \cv J_{ l_v}^{\la_v}(t_v)$. Thus, we have
$G_{l_w}^{\la_w}(w) \subset \cv G_{ l_v}^{\la_v}(v)$.

\ref{iii}: $\mathcal{Q}(v)=G_{l_v}^{\la_v}(v)$ and $ \mathcal{Q}(w)=Q^{\la_w}_{l_w}(w)$. 
For any $\tau \in I_{l_w}^{\la_w}(t_w)$, we have from \eqref{sec6:14} that
\begin{equation} \label{sec6:232}
|\tau - t_v|  \leq |\tau - t_w| + |t_w - t_v|\leq 2 \la_w^{2-p}l_w^2 + \frac{\la_v^2}{H_v(\la_v)} l_v^2 = 2 \la_w^{2-p}l_w^2 + \frac{\la_v^2}{\La} l_v^2
\end{equation}
Recalling $K\la_w^p\ge a(w)\la_w^q$, we apply~\eqref{sec6:181}, $2/p\le1$ and~\eqref{sec6:14} to get 
\begin{align*}
\la_w^{2-p} \leq \frac{2 \la_w^2}{\la_w^p + \tfrac{a(w)}{K} \la_w^q} \leq \frac{2K \la_w^2}{\la_w^p + a(w) \la_w^q} \leq 4K^2 \frac{\la_v^2}{\la_w^p + a(w)\la_w^q}=4K^2 \frac{\la_v^2}{\La},
\end{align*}
which, together with~\eqref{sec6:232} and \eqref{sec6:3_2}, implies
\begin{equation*}
|\tau - t_v| \leq (32K^2+1)\frac{\la_v^2}{\La}l_v^2\le 100K^2\frac{\la_v^2}{\La}l_v^2=\frac{\la_v^2}{\La}(\cv l_v)^2.
\end{equation*}
Therefore $I_{l_w}^{\la_w}(t_w) \subset \cv J_{ l_v}^{\la_v}(t_v)$ and $Q_{l_w}^{\la_w}(w)\subset \cv G_{l_v}^{\la_v}(v)$.

\ref{iv}: $\mathcal{Q}(v)=Q_{l_v}^{\la_v}(v)$ and $\mathcal{Q}(w)=G_{l_w}^{\la_w}(w)$.
For any $\tau \in J_{l_w}^{\la_w}(t_w)$, we apply \eqref{sec6:3_2},~\eqref{sec6:181} and $(p-2)/p\le 1$ to have
\begin{align*}
    |\tau - t_v|&\leq |\tau - t_w| + |t_w - t_v|\leq 2 \frac{\la_w^2}{H_w(\la_w)}l_w^2 + \la_v^{2-p} l_v^2 \\
    &\le  2 \la_w^{2-p}l_w^2 + \la_v^{2-p} l_v^2 \le (16K+1)\la_v^{2-p}l_v^2\\
    &\le 100K^2\la_v^{2-p}l_v^2=\la_v^{2-p}(\cv l_v)^2.
\end{align*}
Therefore $J^{\la_w}_{l_w}(t_w)\subset \cv I_{l_v}^{\la_v}(t_v)$  and $G_{l_w}^{\la_w}(w) \subset \cv Q_{l_v}^{\la_v}(v)$.
Since we have covered every case, the proof of the second condition in \eqref{sec6:1_2} is completed.

\subsection{Final proof of the gradient estimate}
We write the countable pairwise disjoint collection $\mathcal{G}$ defined in \eqref{sec6:7_2} as
$\mathcal{G}=\cup_{j=1}^\infty\mathcal{Q}_j$,
where $\mathcal{Q}_j=\mathcal{Q}(\mz_j)$ with $\mz_j \in \Psi(\Lambda,r_1)$.

Lemma~\ref{sec5:lem:4} and Lemma~\ref{sec5:lem:6} imply that there exist $c=c(\data)$ and $\theta_0=\theta_0(n,p,q)\in(0,1)$ such that
\begin{align*}
		\iint_{\cv\mathcal{Q}_{j}}H(z,|\na u|)\,dz
		\le c\La^{1-\theta}\iint_{\mathcal{Q}_j\cap \Psi(c^{-1}\La)}H(z,|\na u|)^\theta\,dz
		+c\iint_{\mathcal{Q}_j\cap \Phi(c^{-1}\La)}H(z,|F|)\,dz
\end{align*}
for every $j\in\mathbb{N}$ with $\theta= (\theta_0+1)/2$.
By summing over $j$ and applying the fact that the cylinders in $\mathcal{G}$ are pairwise disjoint, we obtain
\begin{equation}\label{sec6:25}
\begin{split}
&\iint_{\Psi(\La,r_1)}H(z,|\na u|)\,dz
\le\sum_{j=1}^\infty\iint_{\cv\mathcal{Q}_{j}}H(z,|\na u|)\,dz\\
&\qquad\le c\La^{1-\theta}\sum_{j=1}^\infty\iint_{\mathcal{Q}_j\cap \Psi(c^{-1}\La)}H(z,|\na u|)^\theta\,dz
+c\sum_{j=1}^\infty\iint_{\mathcal{Q}_j\cap \Phi(c^{-1}\La)}H(z,|F|)\,dz\\
&\qquad\le c\La^{1-\theta}\iint_{\Psi(c^{-1}\La,r_2)}H(z,|\na u|)^\theta\,dz
+c\iint_{\Phi(c^{-1}\La,r_2)}H(z,|F|)\,dz.
\end{split}
\end{equation}
Moreover, since
\[
\iint_{\Psi(c^{-1}\La,r_1)\setminus \Psi(\La,r_1)}H(z,|\na u|)\,dz
\le\La^{1-\theta}\iint_{\Psi(c^{-1}\La,r_2)}H(z,|\na u|)^{\theta}\,dz,
\]
we conclude from \eqref{sec6:25} that
\begin{equation}\label{sec6:26}
	\begin{split}
		&\iint_{\Psi(c^{-1}\La,r_1)}H(z,|\na u|)\,dz\\
		&\qquad\le c\La^{1-\theta}\iint_{\Psi(c^{-1}\La,r_2)}H(z,|\na u|)^\theta\,dz
		+c\iint_{\Phi(c^{-1}\La,r_2)}H(z,|F|)\,dz.
	\end{split}
\end{equation}

For $k\in\mathbb N$, let
\[		
H(z,|\na u|)_k=\min\{H(z,|\na u|),k\}
\]
and
\[
\Psi_k(\La,\rho)=\{z\in Q_{\rho}(z_0):H(z,|\na u(z)|)_k>\La\}.
\]
It is easy to see that if $\La>k$, then $\Psi_k(\La,\rho)=\emptyset$ and if $\La\le k$, then $\Psi_k(\La,\rho)=\Psi(\La,\rho)$. Therefore, we deduce from \eqref{sec6:26} that
\begin{align*}
	\begin{split}
		&\iint_{\Psi_k(c^{-1}\La,r_1)}\left(H(z,|\na u|)_k\right)^{1-\theta}H(z,|\na u|)^\theta\,dz\\
		&\qquad\le c\La^{1-\theta}\iint_{\Psi_k(c^{-1}\La,r_2)}H(z,|\na u|)^\theta\,dz+c\iint_{\Phi(c^{-1}\La,r_2)}H(z,|F|)\,dz.
	\end{split}
\end{align*}

Recalling \eqref{sec6:4}, we denote
\begin{align*}
	\La_1=c^{-1}\left(\frac{4\cv r}{r_2-r_1}\right)^\frac{q(n+2)}{2}\La_0.
\end{align*}
Then for any $\La>\La_1$, we obtain
\begin{align*}
	\begin{split}
		&\iint_{\Psi_k(\La,r_1)}\left(H(z,|\na u|)_k\right)^{1-\theta}H(z,|\na u|)^\theta\,dz\\
		&\qquad\le c\La^{1-\theta}\iint_{\Psi_k(\La,r_2)}H(z,|\na u|)^\theta\,dz+c\iint_{\Phi(\La,r_2)}H(z,|F|)\,dz.
	\end{split}
\end{align*}

Let $\ep\in(0,1)$ to be chosen later. We multiply the inequality above by $\La^{\ep-1}$ and integrate each term over $(\La_1,\infty)$, which implies
\begin{align*}
	\begin{split}
		\mathrm{I}&=\int_{\La_1}^{\infty}\La^{\ep-1}\iint_{\Psi_k(\La,r_1)}\left(H(z,|\na u|)_k\right)^{1-\theta}H(z,|\na u|)^\theta\,dz\,d\La\\
		&\le c\int_{\La_1}^{\infty}\La^{\ep-\theta}\iint_{\Psi_k(\La,r_2)}H(z,|\na u|)^\theta\,dz\,d\La
		+c\int_{\La_1}^{\infty}\La^{\ep-1}\iint_{\Phi(\La,r_2)}H(z,|F|)\,dz\,d\La \\
		&= \mathrm{II}+ \mathrm{III}.
	\end{split}
\end{align*}

We apply Fubini's theorem to estimate $\mathrm{I}$ and obtain
\begin{align*}
	\begin{split}
		\mathrm{I}
		&=\frac{1}{\ep}\iint_{\Psi_k(\La_1,r_1)}\left(H(z,|\na u|)_k\right)^{1-\theta+\ep}H(z,|\na u|)^\theta\,dz\\
		&\qquad-\frac{1}{\ep}\La_1^\ep\iint_{\Psi_k(\La_1,r_1)}\left(H(z,|\na u|)_k\right)^{1-\theta}H(z,|\na u|)^\theta\,dz.
	\end{split}
\end{align*}
Since 
\begin{align*}
	\begin{split}
		&\iint_{Q_{r_1}(z_0)\setminus \Psi_k(\La_1,r_1)}\left(H(z,|\na u|)_k\right)^{1-\theta+\ep}H(z,|\na u|)^\theta\,dz\\
		&\qquad\le \La_1^{\ep}\iint_{Q_{2r}(z_0)}\left(H(z,|\na u|)_k\right)^{1-\theta}H(z,|\na u|)^\theta\,dz,
	\end{split}
\end{align*}
we have
\begin{align*}
	\begin{split}
		\mathrm{I}\ge& \frac{1}{\ep}\iint_{Q_{r_1}(z_0)}\left(H(z,|\na u|)_k\right)^{1-\theta+\ep}H(z,|\na u|)^\theta\,dz\\
		&\qquad-\frac{2}{\ep}\La_1^\ep\iint_{Q_{2r}(z_0)}\left(H(z,|\na u|)_k\right)^{1-\theta}H(z,|\na u|)^\theta\,dz.
	\end{split}
\end{align*}
Similarly, by Fubini's theorem, we have
\begin{align*}
		\mathrm{II}
		\le\frac{1}{1-\theta+\ep}\iint_{Q_{r_2}(z_0)}\left(H(z,|\na u|)_k\right)^{1-\theta+\ep}H(z,|\na u|)^\theta \,dz
\end{align*}
and
\begin{align*}
		\mathrm{III}\le \frac{1}{\ep}\iint_{Q_{2r}(z_0)}H(z,|F|)^{1+\ep}\,dz.
\end{align*}

By combining the estimates above we obtain
\begin{align*}
	\begin{split}
		&\iint_{Q_{r_1}(z_0)}\left(H(z,|\na u|)_k\right)^{1-\theta+\ep}H(z,|\na u|)^\theta\,dz\\
		&\qquad\le \frac{c\ep}{1-\theta+\ep}\iint_{Q_{r_2}(z_0)}\left(H(z,|\na u|)_k\right)^{1-\theta+\ep}H(z,|\na u|)^\theta \,dz\\
		&\qquad\qquad+c\La_1^\ep\iint_{Q_{2r}(z_0)}\left(H(z,|\na u|)_k\right)^{1-\theta}H(z,|\na u|)^\theta\,dz
		+c\iint_{Q_{2r}(z_0)}H(z,|F|)^{1+\ep}\,dz.
	\end{split}
\end{align*}
We choose $\ep_0=\ep_0(\data)\in(0,1)$ so that for any $\ep\in(0,\ep_0)$,
\begin{align*}
	\frac{c\ep}{1-\theta+\ep}\le\frac{1}{2}.
\end{align*}
Then, by applying Lemma~\ref{sec2:lem:2} we get
\begin{align*}
	\begin{split}
		&\iint_{Q_{r}(z_0)}\left(H(z,|\na u|)_k\right)^{1-\theta+\ep}H(z,|\na u|)^\theta\,dz\\
		&\qquad\le c\La_0^\ep\iint_{Q_{2r}(z_0)}\left(H(z,|\na u|)_k\right)^{1-\theta}H(z,|\na u|)^\theta\,dz+c\iint_{Q_{2r}(z_0)}H(z,|F|)^{1+\ep}\,dz.
	\end{split}
\end{align*}
The claim follows by letting $k\longrightarrow\infty$ and recalling \eqref{sec6:01}.

\medskip
\textbf{Acknowledgments.} This research was started during a visit of the first author at the Department of Mathematics of Aalto University. He would like to thank the Nonlinear Partial Differential Equations group for the kind and warm hospitality.

%\section*{References}


\begin{thebibliography}{10}

\bibitem{MR4477803}
R.~Arora and S.~Shmarev.
\newblock Double-phase parabolic equations with variable growth and nonlinear
  sources.
\newblock {\em Adv. Nonlinear Anal.}, 12(1):304--335, 2023.

\bibitem{MR3348922}
P.~Baroni, M.~Colombo, and G.~Mingione.
\newblock Harnack inequalities for double phase functionals.
\newblock {\em Nonlinear Anal.}, 121:206--222, 2015.

\bibitem{MR2779582}
V.~B\"{o}gelein and F.~Duzaar.
\newblock Higher integrability for parabolic systems with non-standard growth
  and degenerate diffusions.
\newblock {\em Publ. Mat.}, 55(1):201--250, 2011.

\bibitem{MR3102165}
V.~B\"{o}gelein, F.~Duzaar, and P.~Marcellini.
\newblock Parabolic equations with {$p,q$}-growth.
\newblock {\em J. Math. Pures Appl. (9)}, 100(4):535--563, 2013.

\bibitem{MR3073153}
V.~B\"{o}gelein, F.~Duzaar, and P.~Marcellini.
\newblock Parabolic systems with {$p,q$}-growth: a variational approach.
\newblock {\em Arch. Ration. Mech. Anal.}, 210(1):219--267, 2013.

\bibitem{MR4487513}
M.~Bul\'{\i}\v{c}ek, P.~Gwiazda, and J.~Skrzeczkowski.
\newblock On a range of exponents for absence of {L}avrentiev phenomenon for
  double phase functionals.
\newblock {\em Arch. Ration. Mech. Anal.}, 246(1):209--240, 2022.

\bibitem{MR3985549}
I.~Chlebicka, P.~Gwiazda, and A.~Zatorska-Goldstein.
\newblock Parabolic equation in time and space dependent anisotropic
  {M}usielak-{O}rlicz spaces in absence of {L}avrentiev's phenomenon.
\newblock {\em Ann. Inst. H. Poincar\'{e} C Anal. Non Lin\'{e}aire},
  36(5):1431--1465, 2019.

\bibitem{MR3294408}
M.~Colombo and G.~Mingione.
\newblock Regularity for double phase variational problems.
\newblock {\em Arch. Ration. Mech. Anal.}, 215(2):443--496, 2015.

\bibitem{MR3447716}
M.~Colombo and G.~Mingione.
\newblock Calder\'{o}n-{Z}ygmund estimates and non-uniformly elliptic
  operators.
\newblock {\em J. Funct. Anal.}, 270(4):1416--1478, 2016.

\bibitem{MR4150873}
C.~De~Filippis.
\newblock Gradient bounds for solutions to irregular parabolic equations with
  {$(p, q)$}-growth.
\newblock {\em Calc. Var. Partial Differential Equations}, 59(5):Paper No. 171,
  32, 2020.

\bibitem{MR3985927}
C.~De~Filippis and G.~Mingione.
\newblock A borderline case of {C}alder\'{o}n-{Z}ygmund estimates for
  nonuniformly elliptic problems.
\newblock {\em Algebra i Analiz}, 31(3):82--115, 2019.

\bibitem{MR1230384}
E.~DiBenedetto.
\newblock {\em Degenerate parabolic equations}.
\newblock Universitext. Springer-Verlag, New York, 1993.

\bibitem{MR2076158}
L.~Esposito, F.~Leonetti, and G.~Mingione.
\newblock Sharp regularity for functionals with {$(p,q)$} growth.
\newblock {\em J. Differential Equations}, 204(1):5--55, 2004.

\bibitem{MR4065088}
F.~Giannetti, A.~Passarelli~di Napoli, and C.~Scheven.
\newblock On higher differentiability of solutions of parabolic systems with
  discontinuous coefficients and {$(p,q)$}-growth.
\newblock {\em Proc. Roy. Soc. Edinburgh Sect. A}, 150(1):419--451, 2020.

\bibitem{MR1962933}
E.~Giusti.
\newblock {\em Direct methods in the calculus of variations}.
\newblock World Scientific Publishing Co., Inc., River Edge, NJ, 2003.

\bibitem{MR4302665}
P.~H\"{a}st\"{o} and J.~Ok.
\newblock Higher integrability for parabolic systems with {O}rlicz growth.
\newblock {\em J. Differential Equations}, 300:925--948, 2021.

\bibitem{MR4397041}
P.~H\"{a}st\"{o} and J.~Ok.
\newblock Maximal regularity for local minimizers of non-autonomous
  functionals.
\newblock {\em J. Eur. Math. Soc. (JEMS)}, 24(4):1285--1334, 2022.

\bibitem{MR4467321}
P.~H\"{a}st\"{o} and J.~Ok.
\newblock Regularity theory for non-autonomous partial differential equations
  without {U}hlenbeck structure.
\newblock {\em Arch. Ration. Mech. Anal.}, 245(3):1401--1436, 2022.

\bibitem{KKS}
W.~Kim, J.~Kinnunen, and L.~Särkiö.
\newblock Lipschitz truncation method for the parabolic double-phase system and applications.
\newblock {\em arXiv}, 2023.

\bibitem{MR1749438}
J.~Kinnunen and J.~L. Lewis.
\newblock Higher integrability for parabolic systems of {$p$}-{L}aplacian type.
\newblock {\em Duke Math. J.}, 102(2):253--271, 2000.

\bibitem{MR969900}
P.~Marcellini.
\newblock Regularity of minimizers of integrals of the calculus of variations
  with nonstandard growth conditions.
\newblock {\em Arch. Rational Mech. Anal.}, 105(3):267--284, 1989.

\bibitem{MR1094446}
P.~Marcellini.
\newblock Regularity and existence of solutions of elliptic equations with
  {$p,q$}-growth conditions.
\newblock {\em J. Differential Equations}, 90(1):1--30, 1991.

\bibitem{MR4258810}
G.~Mingione and V.~R\v{a}dulescu.
\newblock Recent developments in problems with nonstandard growth and
  nonuniform ellipticity.
\newblock {\em J. Math. Anal. Appl.}, 501(1):Paper No. 125197, 41, 2021.

\bibitem{MR1810360}
M.~Rů\v{z}i\v{c}ka.
\newblock {\em Electrorheological fluids: modeling and mathematical theory},
  volume 1748 of {\em Lecture Notes in Mathematics}.
\newblock Springer-Verlag, Berlin, 2000.

\bibitem{MR3532237}
T.~Singer.
\newblock Existence of weak solutions of parabolic systems with {$p,
  q$}-growth.
\newblock {\em Manuscripta Math.}, 151(1-2):87--112, 2016.

\bibitem{MR864171}
V.~V. Zhikov.
\newblock Averaging of functionals of the calculus of variations and elasticity
  theory.
\newblock {\em Izv. Akad. Nauk SSSR Ser. Mat.}, 50(4):675--710, 877, 1986.

\bibitem{MR1209262}
V.~V. Zhikov.
\newblock Lavrentiev phenomenon and homogenization for some variational
  problems.
\newblock {\em C. R. Acad. Sci. Paris S\'{e}r. I Math.}, 316(5):435--439, 1993.

\bibitem{MR1329546}
V.~V. Zhikov, S.~M. Kozlov, and O.~A. Ole\u{\i}nik.
\newblock {\em Homogenization of differential operators and integral
  functionals}.
\newblock Springer-Verlag, Berlin, 1994.
\newblock Translated from the Russian by G. A. Yosifian [G. A. Iosif' yan].

\end{thebibliography}
\end{document}